\documentclass[a4paper,8pt]{article}
\usepackage{amsfonts}
\usepackage{amssymb,amsthm,color,comment,graphicx,tikz}
\usepackage{amsmath}
\usepackage[normalem]{ulem}

\newcommand{\nif}{{n\ra +\infty}}

\newcommand{\R}{\mathbb{R}}
\newcommand{\N}{\mathbb{N}}

\newcommand{\Pa}{\mathbb{P}}

\newcommand{\lra}{\longrightarrow}
\newcommand{\ra}{\rightarrow}
\newcommand{\rau}{\rightarrow}

\newtheorem{theorem}{Theorem}
\newtheorem{lemma}[theorem]{Lemma}

\newtheorem{example}[theorem]{Example}
\newtheorem{definition}{Definition}

\def\Rr{\mathcal R}
\def\S{{\mathcal S}}

\newcommand{\Om}{\Omega}
\newcommand{\lb}{\lambda}

\newcommand{\sq}{\subseteq}

\newcommand{\vphi}{\varphi}
\newcommand{\vps}{\varepsilon}

\newcommand{\bp}{\begin{proof}}
\newcommand{\ep}{\end{proof}}

\begin{document}
\title{On localisation of eigenfunctions of the Laplace operator}

\author{{M. van den Berg\footnote{corresponding author}} \\
School of Mathematics, University of Bristol\\
Fry Building, Woodland Road\\
Bristol BS8 1UG\\
United Kingdom\\
\texttt{mamvdb@bristol.ac.uk}\\
\\
{D. Bucur}\\
Laboratoire de Math\'ematiques, Universit\'e Savoie Mont Blanc \\
UMR CNRS  5127\\
Campus Scientifique,
73376 Le-Bourget-Du-Lac\\
France\\
\texttt{dorin.bucur@univ-savoie.fr}}
\date{ 27 January 2025}\maketitle
\vskip 1.0truecm \indent

\begin{center}{ This paper is dedicated to the memory of our friend and colleague\\ Thomas Kappeler.}\end{center}

\begin{abstract}\noindent
We prove (i) a simple sufficient geometric condition for localisation of a sequence of first Dirichlet eigenfunctions provided the corresponding Dirichlet Laplacians satisfy a uniform Hardy inequality, and (ii) localisation of a sequence of first Dirichlet eigenfunctions for a wide class of elongating horn-shaped domains. We give examples of sequences of simply connected, planar, polygonal domains for which the corresponding sequence of first  eigenfunctions with either Dirichlet, or Neumann, boundary conditions $\kappa$-localise in $L^2$.
\end{abstract}
\begin{center} \textbf{Keywords}: First Dirichlet eigenfunction, localisation, $\kappa$-localisation, Hardy inequality.
\end{center}
\mbox{}

\section{Introduction\label{sec1}}

In this paper we study the phenomenon of localisation for eigenfunctions of the Laplace operator for domains in Euclidean space. Let $\Om$ be a non-empty open, bounded and connected set in $\R^m$ with Lebesgue measure $|\Om|$. The spectrum of the Dirichlet Laplacian acting in $L^2(\Om)$ is discrete, and consists of eigenvalues $\{\lambda_1(\Om)\le \lambda_2(\Om)\le...\}$ accumulating at infinity only. We denote a corresponding orthonormal sequence of Dirichlet eigenfunctions by $\{u_{j,\Om},\,j\in \N\}$. Throughout we denote the {$L^p$ norm, $1\le p\le \infty$, by $\|\cdot\|_p$.}  Since $\Om$ is connected the first eigenvalue is simple, and the corresponding eigenspace is one-dimensional. The corresponding eigenfunction is determined up to a sign, and we choose $u_{1,\Om}>0$, and write $u_{\Om}:=u_{1,\Om}$. The question of localisation is the following.
Does there exist, given a small $\varepsilon\in (0,1)$, a measurable set $A_{\vps}\subset \Om$ with
\begin{equation}\label{ea}
\frac{|A_{\vps}|}{|\Om|}\le\varepsilon,\,\, \int_{A_{\vps}}u^2_{\Om}\ge 1-\varepsilon.
\end{equation}
If \eqref{ea} holds, then
\begin{equation}\label{ez}
1-\vps\le \|u_{\Om}\|^2_{\infty}|A_{\vps}|\le \varepsilon\|u_{\Om}\|^2_{\infty}||\Om|.
\end{equation}
{We recall that (see \cite[equation (26)]{vdB})}
\begin{equation}\label{eh}
\|u_{\Om}\|_{\infty}\le \bigg(\frac{e}{2\pi m}\bigg)^{m/4}\lambda_1(\Om)^{m/4}
\end{equation}
By \eqref{ez} and \eqref{eh} we have that
\begin{equation}\label{ej}
\lambda_1(\Om)|\Om|^{2/m}\ge \frac{2\pi m}{e} \Big(\frac{1-\vps}{\vps}\Big)^{2/m},
\end{equation}
and the first eigenvalue is, for small $\vps$, large compared with the Faber-Krahn lower bound. The latter states that
\begin{equation*}
\lambda_1(\Om)|\Om|^{2/m}\ge \lambda_1(B_1)|B_1|^{2/m},
\end{equation*}
where $B_1$ is an open ball with radius $1$.

The torsion function for an open set $\Om,\,0<|\Om|<\infty$ is the unique solution of
\begin{equation*}
-\Delta v=1,\, \qquad v\in H_0^{ 1 }(\Omega),
\end{equation*}
and is denoted by $v_{\Om}$. The torsion function is non-negative, bounded and monotone under set inclusion. A much studied quantity is the torsional rigidity, defined by
$$T(\Om)=\int_{\Om}v_{\Om}. $$
See, for example, \cite{LB} and some of the references therein.
It turns out that the localisation question for the torsion function stated below in $L^1$ is closely related to the localisation question for the first Dirichlet {eigenfunction in $L^2$} (see the paragraph above \eqref{e43}). Does there exist, given a small $\varepsilon\in (0,1)$, a measurable set $A_{\vps}\subset \Om$ with
\begin{equation}\label{eb}
\frac{|A_{\vps}|}{|\Om|}\le\varepsilon,\,\, \frac{\int_{A_{\vps}}v_{\Om}}{\int_{\Om}v_{\Om}}\ge 1-\varepsilon.
\end{equation}
If there exists $A_{\vps}$ satisfying \eqref{eb}, then
\begin{equation}\label{ek}
1-\vps\le \int_{A_{\vps}}\frac{v_{\Om}}{T(\Om)}\le T(\Om)^{-1}\|v_{\Om}\|_{\infty}|A_{\vps}|\le T(\Om)^{-1}\|v_{\Om}\|_{\infty}|\Om|\vps.
\end{equation}
By \cite[Theorem 1]{vdBC}, we have
\begin{equation}\label{el}
\|v_{\Om}\|_{\infty}\le (4+3m\log 2)\lambda_1(\Om)^{-1},
\end{equation}
and by the Kohler-Jobin inequality (see \cite{KJ1}, \cite{KJ2}), we have
\begin{equation}\label{em}
T(\Om)\lambda_1(\Om)^{(m+2)/2}\ge T(B_1)\lambda_1(B_1)^{(m+2)/2}.
\end{equation}
We find that
\begin{equation}\label{en}
\lambda_1(\Om)|\Om|^{2/m}\ge K_m\Big(\frac{1-\vps}{\vps}\Big)^{2/m},
\end{equation}
where $K_m>0$ can be read-off from \eqref{ek}, \eqref{el} and \eqref{em}.
Again we see that if there exists $A_{\vps}$ satisfying \eqref{eb}, then
the first eigenvalue is, for small $\vps$, large compared with the Faber-Krahn lower bound.

To simplify the discussion we define localisation for sequences. Let $p\in[1,\infty)$ be fixed, and let $(\Om_n)$ be a sequence of open sets in $\R^m$ with $0<|\Om_n|<\infty,n\in\N$. For $n\in \N$, let $f_n\in L^p(\Omega_n),\,0<\|f_n\|_p<\infty$.
Define the following collection of sequences
\begin{equation*}
\mathfrak{A}((\Omega_n))=\bigg\{(A_n): (\forall n\in \N)(A_n\subset\Omega_n, A_n\, \textup{measurable}), \lim_{n\rightarrow\infty}\frac{|A_n|}{|\Omega_n|}=0\bigg\},
\end{equation*}
and let
\begin{equation}\label{e2}
\kappa=\sup\bigg\{\limsup_{n\rightarrow\infty}\frac{\|f_n{\bf 1}_{A_n}\|_p^p}{\|f_n\|_p^p}:(A_n)\in\mathfrak{A}((\Omega_n))\bigg\},
\end{equation}
where ${\bf 1}_{.}$ is the indicator function. Note that $0\le \kappa\le 1$.

We write $(f_n)$ for the sequence of functions $f_n:\Om_n\rightarrow \R, n\in \N$ in the following definition (\cite{vdBK}).

\noindent\begin{definition}\label{def1} We say that
\begin{enumerate}
\item[\textup{(i)}] the sequence $\big(f_{n}\big)$ $\kappa$-localises in $L^p$ if $0<\kappa<1$,
\item[\textup{(ii)}] the sequence $\big(f_{n}\big)$ localises in $L^p$ if $\kappa=1$,
\item[\textup{(iii)}]the sequence $\big(f_{n}\big)$ does not localise in $L^p$ if $\kappa=0$.
\end{enumerate}
\end{definition}

We see that, using Cantor's diagonalisation procedure, the supremum in \eqref{e2} is achieved by a maximising sequence. Let $(A_n)$ be such a sequence. This sequence is not unique since modification by sets of measure $0$ does not change $\kappa$.

For $p=2$ and $f_n=u_{\Om_n}$, Definition \ref{def1}(ii) is equivalent to the following.
There exist sequences $(\varepsilon_n)$ with $ \lim_{n\rightarrow\infty}\vps_n=0$, and  $(A_{n})\in \mathfrak{A}((\Om_n))$ such that
\begin{equation}\label{ec}
\frac{|A_{n}|}{|\Om_n|}\le\varepsilon_n,\,\, \int_{A_{n}}u^2_{\Om_n}\ge 1-\varepsilon_n.
\end{equation}
Similarly for $p=1$ and $f_n=v_{\Om_n}$, Definition \ref{def1}(ii) is equivalent to the following.
There exist sequences $(\varepsilon_n)$ with $ \lim_{n\rightarrow\infty}\vps_n=0$, and  $(A_{n})\in \mathfrak{A}((\Om_n))$ such that
\begin{equation}\label{ey}
\frac{|A_{n}|}{|\Om_n|}\le\varepsilon_n,\,\, \frac{\int_{A_{n}}v_{\Om_n}}{\int_{\Om_n}v_{\Om_n}}\ge 1-\varepsilon_n.
\end{equation}
We conclude that if either $(u_{\Om_n})$ localises in $L^2$ or $(v_{\Om_n})$ localises in $L^1$ then, by \eqref{ec} and \eqref{ej}, or \eqref{ey} and \eqref{en},
\begin{equation}\label{ez}
\lim_{n\rightarrow\infty}\lambda_1(\Om_n)|\Om_n|^{2/m}=+\infty.
\end{equation}
We arrive at the same conclusion in the case of $\kappa$-localisation, by replacing $1-\vps$ by $\kappa(1-\vps)$ in the lines above.
On the other hand, by considering a sequence of rectangles $(R_n)$, $R_n=(0,1)\times (0,n)\sq \R^2$ we see that \eqref{ez} is clearly not sufficient for localisation of $(u_{R_n})$ or of $(v_{\Om_n})$.

It was shown in \cite[Theorem 4]{vdBBK} that if $(v_{\Om_n})$ either localises or $\kappa$-localises in $L^1$, then the corresponding sequence of  eigenfunctions $(u_{\Om_n})$ localises in $L^2$. It was pointed out below  \cite[Theorem 4]{vdBBK} that the torsion function does not localise for sequences of convex sets, while it was shown in \cite{vdBDPDBG} that there is a wide class of open, bounded, convex, elongating sequences of sets in $\R^m$ for which the sequence  of first Dirichlet eigenfunctions localises.
See \cite[Examples 8,9,10]{vdBDPDBG}. In \cite[Example 10]{vdBDPDBG} it was shown that the sequence $(u_{\Om_{n,\alpha}})$ localises in $L^2$, where

\begin{align}\label{e43}
\Omega_{n,\alpha}=\big\{&(x_1,x')\in \R^m: -2^{-1}n<x_1<2^{-1}n,\,\big(2n^{-1}|x_1|\big)^{\alpha}+|x'|^{\alpha}<1\big\},\, n\in \N,
\end{align}
and where $\alpha\in[1,\infty)$ is fixed.
{ The following localisation lemma (\cite[Lemma 3]{vdBDPDBG}) plays a crucial role in the proof of Theorem \ref{the2} below.}

\begin{lemma}\label{lem0}
For $n\in \N$, let $f_n\in L^2(\Omega_n)$ with $\|f_{n}\|_2>0,$ and $|\Omega_n|<\infty$. Then $\big(f_{n}\big)$ localises in $L^2$ if and only if
\begin{equation*}
\lim_{n\rightarrow\infty}\frac{1}{|\Om_n|}\frac{\|f_n\|_1^2}{\|f_n\|_2^2}=0.
\end{equation*}
\end{lemma}
Lemma \ref{lem0} shows that a vanishing $L^1$-$L^2$ participation ratio is equivalent to localisation.

Definition \ref{def1} above was motivated by \eqref{ea} and \eqref{eb}.
We note that the very general definition of localisation above, or alternatively vanishing $L^1$-$L^2$ participation ratio in case $p=2$, does not provide any information on where these sequences localise. However, in some  concrete examples, such as in Example \ref{exa1} below, it is possible to obtain this information.

Other ratios have been defined in \cite[equations (7.1)-(7.3)]{DG}. We define the $L^p$-$L^q$ with $p<q$ participation ratio of a function $u\in L^p(\Om)\cap L^q(\Om)$ as the number
$|\Om|^{\frac{1}{q}-\frac{1}{p}}\frac{\|u\|_p}{\|u\|_q}.$
It was shown in \cite{TB} that for $\Om \sq \R^m$ convex, there exist constants $k_m<\infty$ depending on $m$ only such that
\begin{equation}\label{ch1}
\|u_{\Om}\|_{\infty}\le k_m\Big(\frac{\rho(\Om)}{\textup{diam}(\Om)}\Big)^{1/6}\rho(\Om)^{-m/2} \|u_{\Om}\|_2,
\end{equation}
where $\rho(\Om)$ denotes the inradius of $\Omega$ and $\textup{diam}(\Om)$ its diameter.
It follows by \eqref{ch1} that the $L^2$-$L^{\infty}$ ratio is bounded from below by
\begin{equation}\label{t1}
\frac{1}{|\Om|^{1/2}}\frac{\|u_{\Om}\|_2}{\|u_{\Om}\|_{\infty}}\ge k_m^{-1}\Big(\frac{\textup{diam}(\Om)}{\rho(\Om)}\Big)^{1/6}\frac{\rho(\Om)^{m/2}}{|\Om|^{1/2}}.
\end{equation}
In order to get an upper bound for $|\Om|$ in terms of $\textup{diam}(\Om)$ and $\rho(\Om)$ we use John's ellipsoid theorem (\cite{J}). The latter asserts that if $\Om\subset \R^d$ is convex then there exists an ellipsoid $E(a)$ with semi-axes $(a_1,a_2,...,a_m)$ such that $E(a)\subset \Om\subset E(ma)$, where $E(ma)$ is a homothety of $E(a)$ with respect to its centre by a factor $m$. The ellipsoid $E(a)$ is of maximal measure. We may assume, by relabelling the axes, that $a_1\ge a_2\ge ...\ge a_m$.
Hence $a_m\le \rho(\Om)\le ma_m$, and $2a_1\le \textup{diam}(\Om)\le 2ma_1$. It follows that
\begin{align}\label{t2}
|\Om|&\le \omega_mm^ma_1a_2...a_m\le \omega_m m^ma_1^{m-1}a_m\nonumber\\&\le \frac{\omega_mm^m}{2^{m-1}}\textup{diam}(\Om)^{m-1}\rho(\Om).
\end{align}
By \eqref{t1} and \eqref{t2} there exists $\tilde{k}_m$ such that
\begin{equation}\label{t3}
\frac{1}{|\Om|^{1/2}}\frac{\|u_{\Om}\|_2}{\|u_{\Om}\|_{\infty}}\ge \tilde{k}_m\Big(\frac{\rho(\Om)}{\textup{diam}(\Om)}\Big)^{(3m-4)/6}.
\end{equation}
If $(u_{\Om_n})$ localises in $L^2$ then for $\varepsilon\in (0,1)$ and all $n$ sufficiently large, we have by \eqref{ec} that
\begin{equation}\label{t4}
1-\varepsilon\le \int_{A_n}u^2_{\Om_n}\le |A_n|\|u_{\Om_n}\|_{\infty}^2\le \varepsilon|\Om_n|\|u_{\Om_n}\|_{\infty}^2.
\end{equation}
If moreover $\Om_n$ are convex, then by \eqref{t3} and \eqref{t4} we have for all $n$ sufficiently large,
\begin{equation*}
\frac{\rho(\Om_n)}{\textup{diam}(\Om_n)}\le L_m\Big(\frac{\varepsilon}{1-\varepsilon}\Big)^{6/(3m-4)},
\end{equation*}
for some finite $m$-dependent constant $L_m$.
This quantifies the elongation referred to in \eqref{ej} and \eqref{en}.

\smallskip

The rich interplay between localisation and the inverse of the torsion function has been studied in \cite{ADFJM}, \cite{DFM}, and the references therein.

\medskip

The main results of this paper are the following. In Section \ref{sec2} we construct a sequence of simply connected, planar, polygonal domains for which the corresponding sequence of first Dirichlet eigenfunctions $\kappa$-localises in $L^2$ (see as well \cite{SS} for a recent analysis of the eigenfunction localisation on dumbbell domains). In Section \ref{sec4} we prove a simple sufficient geometric condition for localisation of a sequence of first Dirichlet eigenfunctions provided the corresponding Dirichlet Laplacians satisfy a uniform strong Hardy inequality. In Section \ref{sec5} we prove localisation for a wide class of elongating horn-shaped domains. In the case of symmetric two-sided horn-shaped domains we give a sufficient condition  for localisation of the second Dirichlet eigenfunction. The results in that section vastly improve those presented in Theorem 6 and the various examples in \cite{vdBDPDBG}. In particular, no convexity hypotheses are made in Theorem \ref{the2} below. In Section \ref{sec3} we construct a sequence of simply connected, planar, polygonal domains for which the first non-trivial Neumann eigenfunction $\kappa$-localises in $L^2$.

\section{Example of $\kappa$-localisation for Dirichlet eigenfunctions\label{sec2}}

In this section we construct a sequence of simply connected, planar, polygonal domains for which the corresponding sequence of first Dirichlet eigenfunctions $\kappa$-localises in $L^2$.

Let $\vps\in (0,1), \delta>0$ and let $\theta\in(0,\delta)$. Define the following planar open sets.
The rectangle
\begin{equation*}\label{e4}
{R}_\vps=(-\vps, \vps) \times (-\vps^{-1},\vps^{-1}),
\end{equation*}
so that
\begin{equation*}\label{e5}
\lb_1({R}_\vps)= \frac{\pi^2}{4}(\vps^2+ \vps^{-2}).
\end{equation*}
The thin rectangle
\begin{equation*}\label{e8}
T_\theta=(0,2) \times (-\theta, \theta).
\end{equation*}
The square
\begin{equation*}\label{e6}
{S}_\delta=(2-\delta, 2+\delta)\times (-\delta, \delta),
\end{equation*}
so that
\begin{equation*}\label{e7}
\lb_1({S}_\delta)= \frac{\pi^2}{2\delta^2}.
\end{equation*}
The values of $\delta,\theta$ and $\vps$ will be chosen such that $\lb_1({S}_\delta)\approx \lb_1({ R}_\vps)$ and $\theta <\hskip -.15cm< \vps$.

Let
\begin{equation*}
\Om_{\vps, \theta, \delta} = {R}_{\vps}\cup T_{\theta}\cup S_{\delta}.
\end{equation*}
See Figure 1.

%

\begin{figure}
\centering
\begin{tikzpicture}
\draw[black, thick](0,0)--(1,0);
\draw[black, thick](0,0)--(0.0,4);
\draw[black, thick](0,4)--(1,4);
\draw[black, thick](1,4)--(1,2.1);
\draw[black, thick](1,0)--(1,1.9);
\draw[black, thick](1,2.1)--(5,2.1);
\draw[black, thick](1,1.9)--(5,1.9);
\draw[black, thick](5,01.9)--(5,1);
\draw[black, thick](5,2.1)--(5,3);
\draw[black, thick](5,3)--(7,3);
\draw[black, thick](7,3)--(7,1);
\draw[black, thick](7,1)--(5,1);
\node at (0.5,2)    {$R_{\vps}$};
\node at (6,2)    {$S_{\delta}$};
\node at (3,2.3)    {$T_{\theta}$};
\end{tikzpicture}
\caption{$\Omega_{\vps,\theta,\delta}=R_{\vps}\cup T_{\theta}\cup S_{\delta}$}
\label{fig1}
\end{figure}
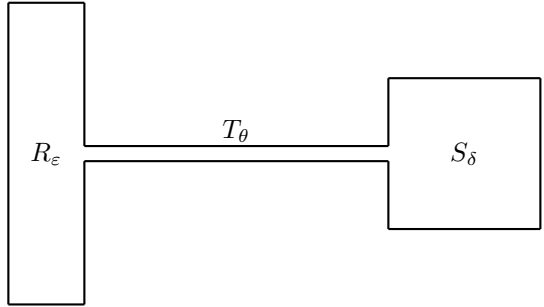

Since $\Om_{\vps, \theta, \delta}$ is connected, $\lb_1(\Om_{\vps, \theta, \delta})$ is simple. Let $u_{\Om_{\vps, \theta, \delta}}$ be the corresponding positive, $L^2$-normalised eigenfunction.

\begin{theorem} Let $\kappa\in (0,1)$ be fixed.  There exists a sequence of sets of the form $\Om_{\vps, \theta, \delta}$ for which the first Dirichlet eigenfunction $\kappa$-localises.
\end{theorem}

\begin{proof}
{\bf Step 1.}  Fix $\vps >0,$ and choose
\begin{equation}\label{e10.1}
\delta_\vps= \frac{\sqrt{2} \vps}{\sqrt{1+\vps^4}}.
\end{equation}
Then
\begin{equation*}
\lb_1({ S}_{\delta_{\vps}})= \lb_1({ R}_\vps)= \frac{\pi^2}{2\delta_\vps^2}=\frac{\pi^2}{4}(\vps^2+ {\vps^{-2}}).
\end{equation*}

\medskip
\noindent
{\bf Step 2.} For $n \in \N$, $n \ge \frac {4}{\delta_\vps}$, $\delta \in [\delta_\vps-\frac 1n, \delta_\vps+\frac 1n]$, and $\theta \in (0,  \frac {\delta_\vps}{4})$ we define
\begin{equation*}\label{e12}
F(\theta, \delta)= \int_{T_\theta \cup S_\delta} u^2_{{\Om}_{\vps, \theta, \delta}}.
\end{equation*}
Since  $\Om_{\vps, \theta, \delta}$ is simply connected, the perturbation of the parameters $\theta$ and $\delta$ is $\gamma$-continuous (see for instance \cite[Chapter 4]{BB05}). Hence $F$ is continuous on
\begin{equation*}\label{e13}
(0,  \delta_\vps/4)\times  [\delta_\vps-n^{-1}, \delta_\vps+n^{-1}].
\end{equation*}
Moreover, we observe that
\begin{equation*}\label{e14}
\lim_{\theta \downarrow 0} F(\theta, \delta_\vps-n^{-1})=0,
\end{equation*}
and
\begin{equation*}\label{e15}
\lim_{\theta \downarrow 0} F(\theta, \delta_\vps+n^{-1})=1.
\end{equation*}
Setting $ \eta =\eta_{n,\vps}:=\frac 12\min\{\frac 1n,  \frac {\delta_\vps}{4}\}= \frac{1}{2n}$, we define the curve $C_\eta :[0,\pi]\rightarrow \R^2$ by
\begin{equation*}\label{e16}
C_\eta (t)= \Big(\eta \sin t,  \delta_\vps-\frac 1n+ \frac{2t}{\pi n}\Big),\,0\le t \le \pi.
\end{equation*}
The function $F$ is continuous along $C_\eta$ and takes the value $0$ at $t=0$ and $1$ at $t=\pi$. By continuity there exists $t^* \in (0, \pi)$ such that
\begin{equation*}\label{e17}
F(C_{\eta} (t^*))= \kappa.
\end{equation*}
Let $C_\eta(t^*)=(\theta_{n, \vps}, \delta_{n, \vps})$.

\medskip
\noindent
{\bf Step 3.} In this step, we keep $\vps$ constant, and let $n \rightarrow +\infty$. We have that
\begin{equation*}\label{e18}
\Om_{\vps, \theta_{n, \vps}, \delta_{n, \vps}} \stackrel{\gamma}{\lra}  { R}_\vps\cup{ S}_{\delta_\vps}
\end{equation*} $\gamma$-converges.
We get
\begin{equation*}\label{e19}
\lim_{n\rightarrow\infty}\lb_1(\Om_{\vps, \theta_{n, \vps}, \delta_{n, \vps}})= \lb_1( { R}_\vps\cup{ S}_{\delta_\vps})= \frac{\pi^2}{4}(\vps^2+  \vps^{-2}).
\end{equation*}
Moreover $u_{\Om_{\vps, \theta_{n, \vps}, \delta_{n, \vps}}}$ converges in $H^1(\R^2)$ to an eigenfunction $u\in H^1_0( { R}_\vps\cup{ S}_{\delta_\vps})$ corresponding to the first eigenvalue of $ { R}_\vps\cup{ S}_{\delta_\vps}$. By our choice of $t^*$ we get
\begin{equation}\label{e21}
\int_{S_{\delta_\vps}} u^2  = \kappa,\quad \int _{{R}_\vps} u^2  =1-\kappa.
\end{equation}
We now keep track of the $L^\infty$-norm of $u_{\Om_{\vps, \theta_{n, \vps}, \delta_{n, \vps}}}$ on ${ R}_\vps$, and claim that
\begin{equation}\label{e22}
\lim_{n \rightarrow +\infty}\| u_{\Om_{\vps, \theta_{n, \vps}, \delta_{n, \vps}}}\|_{L^\infty({ R}_\vps)} = \|u\|_{L^\infty({ R}_\vps)}.
\end{equation}
By the a.e. pointwise convergence we have that
\begin{equation*}\label{e23}
\|u\|_{L^\infty({ R}_\vps)} \le \liminf _{n \rightarrow +\infty}\| u_{\Om_{\vps, \theta_{n, \vps}, \delta_{n, \vps}}}\|_{L^\infty({ R}_\vps)}.
\end{equation*}
In order to prove the converse inequality, we follow a classical strategy (see for instance \cite[Theorem 2.2]{HLP} or \cite{vdBBK}, and the references therein). From the eigenvalue monotonicity with respect to inclusions we obtain by \eqref{eh}
\begin{equation*}\label{e24}
-\Delta u_{\Om_{\vps, \theta_{n, \vps}, \delta_{n, \vps}}} \le  \lb_1^{3/2}{(S_{\delta_{n,\vps}/4})} := M_\vps \mbox { in } {\mathcal D}'(\R^2).
\end{equation*}
Then for every point $x_n \in \R^2$ we get
\begin{equation*}\label{e25}
-\Delta \bigg(u_{\Om_{\vps, \theta_{n, \vps}, \delta_{n, \vps}}}  + M_\vps \frac{|\cdot -x_n|^{2}}{4}\bigg) \le 0 \mbox { in } {\mathcal D}'(\R^2).
\end{equation*}
So by subharmonicity
\begin{equation*}\label{e26}
u_{\Om_{\vps, \theta_{n, \vps}, \delta_{n, \vps}}}(x_n) \le \frac{\int_{B(x_n; \rho)}dx\,\big(u_{\Om_{\vps, \theta_{n, \vps}, \delta_{n, \vps}} }(x) + {M_\vps \frac{|x -x_n|^2}{4}}\big)}{|B(x_n; \rho)|},
\end{equation*}
where $B(p;r)=\{x\in \R^m:|p-x|<r\}$ for $p\in\R^m, r>0.$
Let $x_n\in {R}_\vps  $ be such that
\begin{equation*}\label{e27}
\| u_{\Om_{\vps, \theta_{n, \vps}, \delta_{n, \vps}}}\|_{L^\infty({ R}_\vps)} -\frac 1n \le u_{\Om_{\vps, \theta_{n, \vps}, \delta_{n, \vps}}}(x_n).
\end{equation*}
Taking the limit $n \rightarrow +\infty$, and assuming without loss of generality that $x_n \rightarrow x_0$, we get
\begin{align*}
\limsup_{n \rightarrow +\infty} \| u_{\Om_{\vps, \theta_{n, \vps}, \delta_{n, \vps}}}\|_{L^\infty({ R}_\vps)} &\le  \frac{\int_{B(x_0; \rho)}dx\,\big(u(x)  + M_\vps \frac{|x -x_0|^{2}}{4}\big) }{|B(x_0; \rho)|}\nonumber\\&\le \|u\|_{L^\infty({ R}_\vps)} + M_\vps\frac{\rho^2}{8}.
\end{align*}
Taking the limit $\rho \downarrow 0$ we obtain \eqref{e22}.

Since $u$ is a first eigenfunction on ${R}_\vps$ we have that $\frac{\|u\|_\infty}{\|u\|_2}= \frac{2}{|{R}_\vps|^\frac 12}$.  Consequently, from \eqref{e21} we get
\begin{equation*}\label{e29}
\|u\|_{L^\infty({ R}_\vps)}= 2\sqrt{1-\kappa}.
\end{equation*}

\medskip
\noindent
{\bf Step 4.} Now let $\vps \downarrow 0$. For every such $\vps$, we pick up from Step 3 some $n=n_\vps$ such that
\begin{equation}\label{e30}
\| u_{\Om_{\vps, \theta_{n_\vps}, \delta_{n_\vps}}}\|_{L^\infty({ R}_\vps)}\le 2\sqrt{1-\kappa} + \vps.
\end{equation}
This sequence $\kappa$-localises on $T_{\theta_{n, \vps}} \cup { S} _{ \delta_{n, \vps}}$.
\end{proof}

The data in Figure    2   have been obtained with the MATLAB  PDE toolbox and illustrate the mass distribution of the first eigenfunction.
\begin{figure}\label{fig1vbb}
\centering
\hskip -1.2cm
\includegraphics[width=4cm]{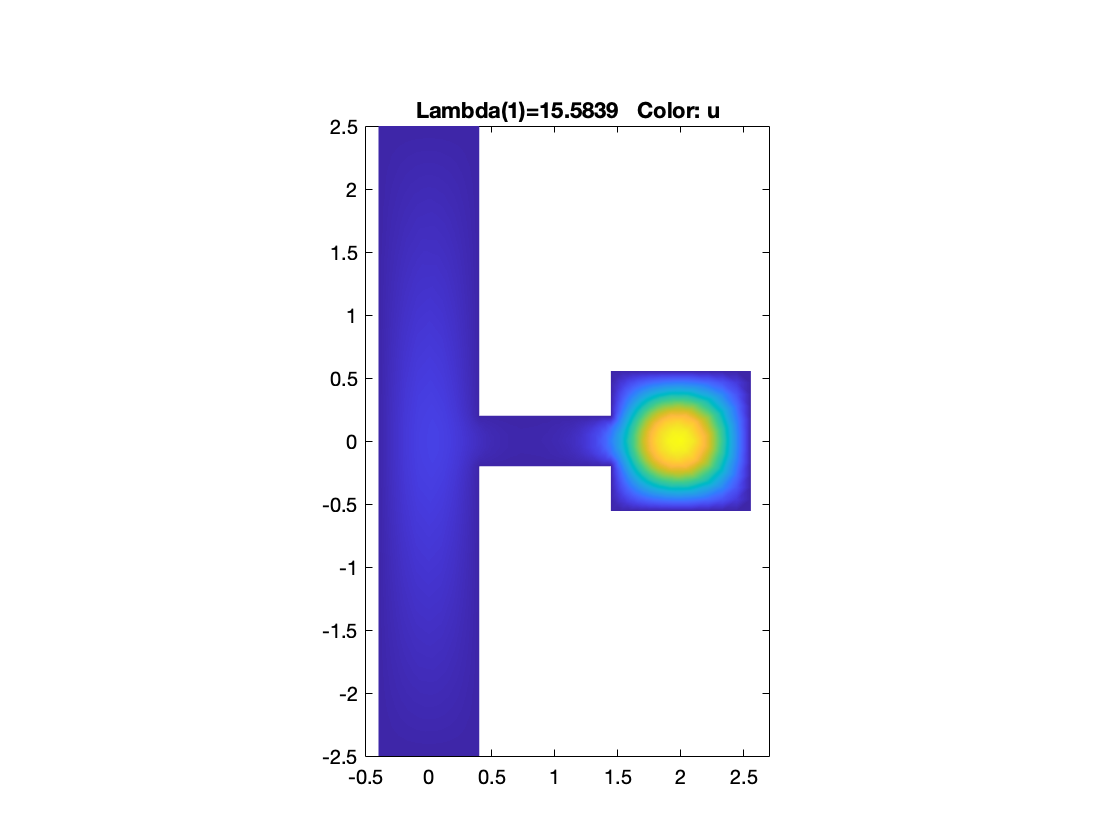}\hskip -1.2cm
\includegraphics[width=4cm]{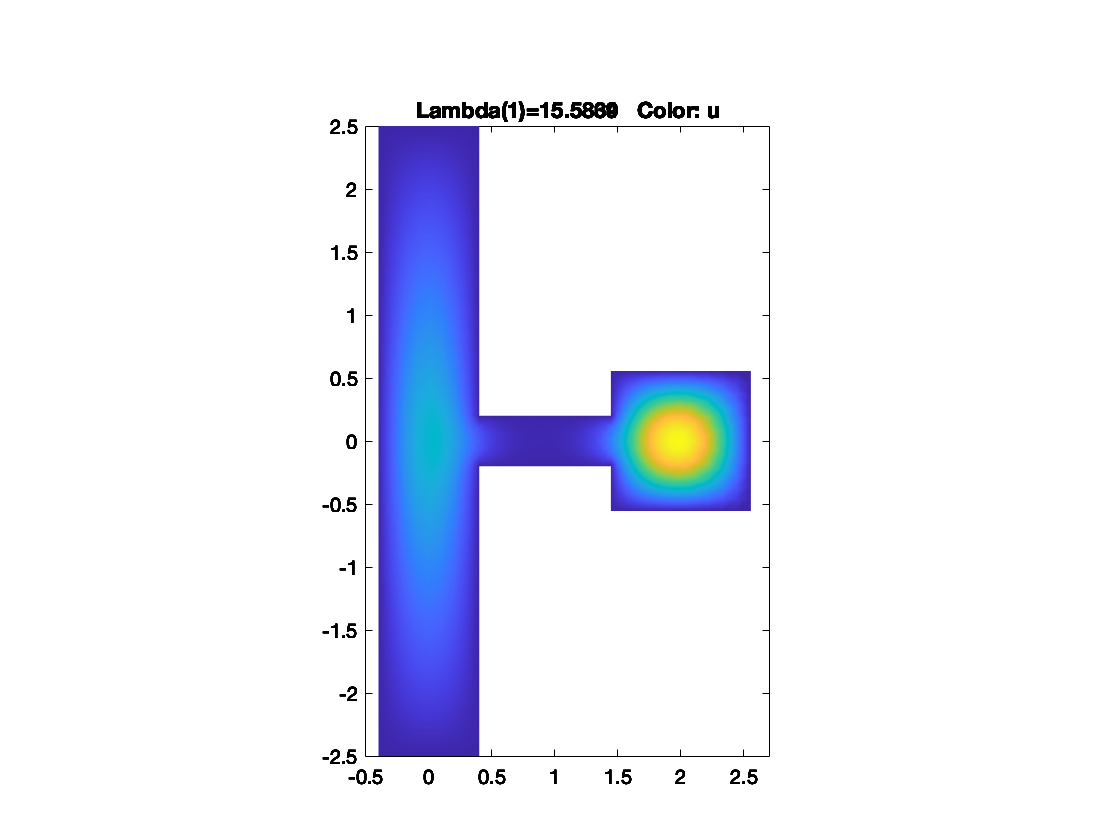}\hskip -1.2cm
\includegraphics[width=4cm]{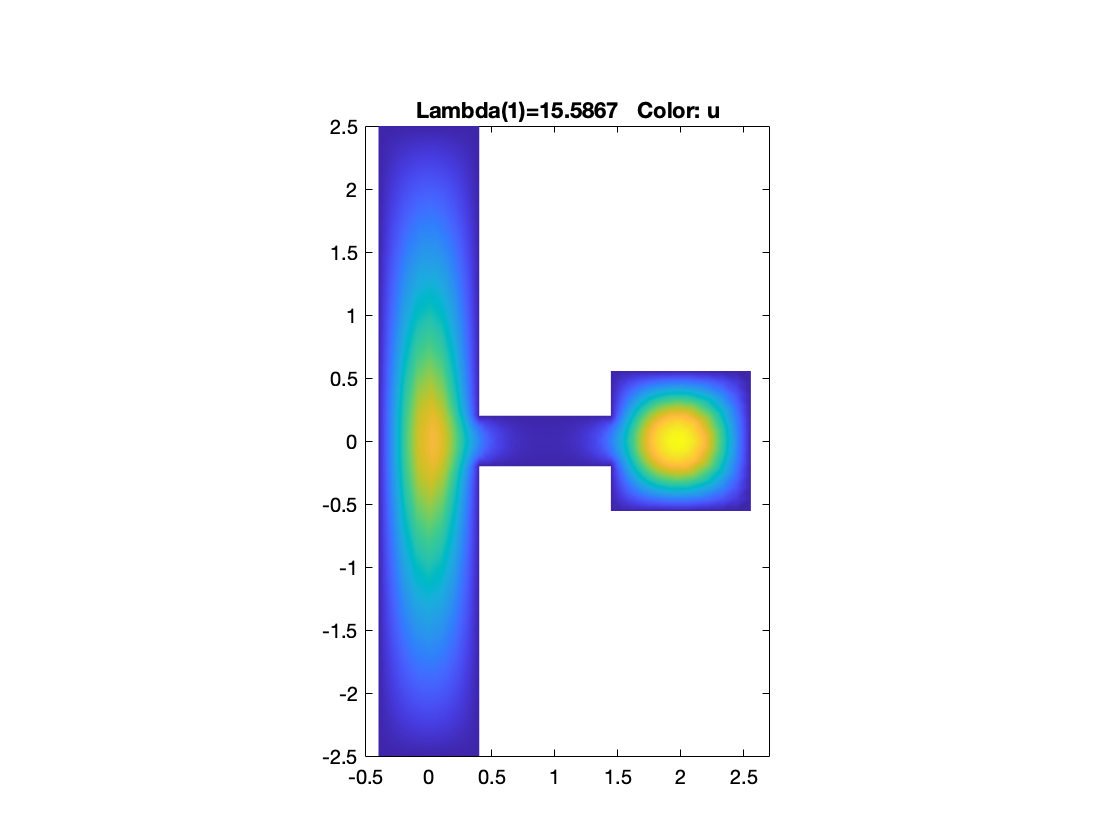}\hskip -1.2cm
\includegraphics[width=4cm]{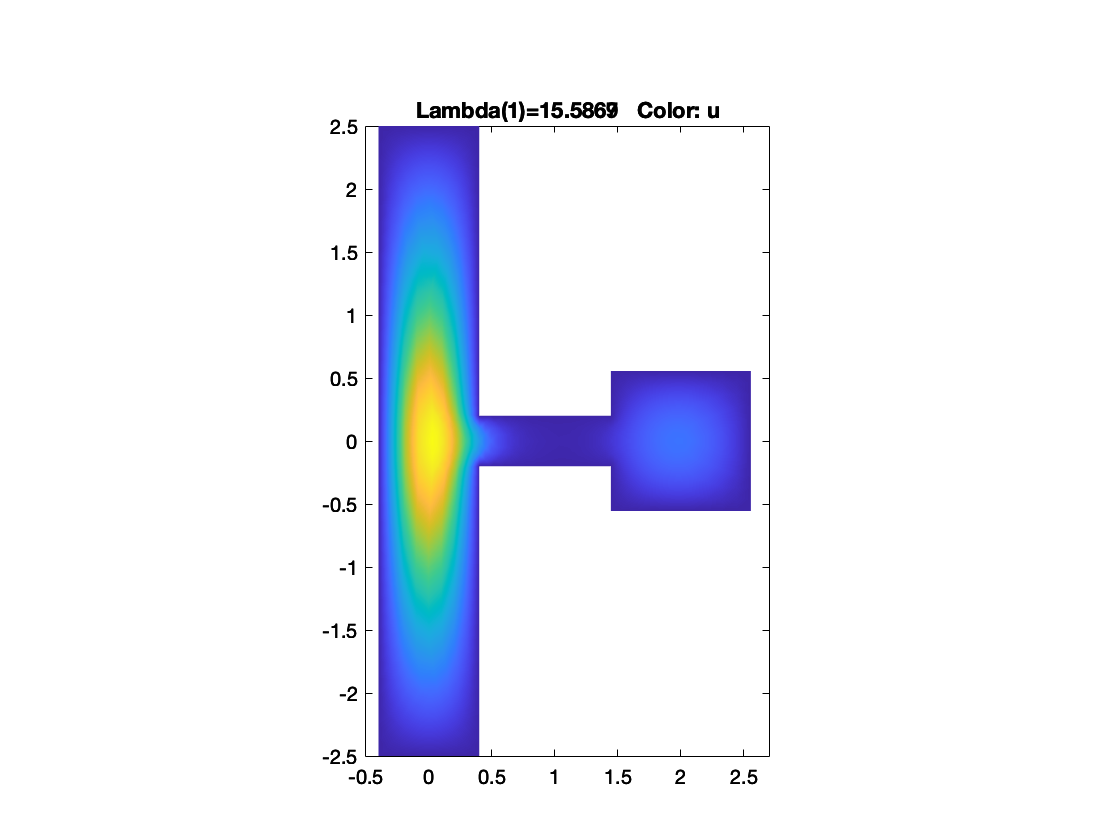}
\caption{The mass distribution of $u_1$ when perturbing the size of the square on the right: $\vps=0.4$, $\theta=0.2$, $\delta=\frac{\sqrt{2} \vps}{\sqrt{1+\vps^4}}-c$, for $c=0.00281$, $c=0.00286$, $c=0.00287$,
$c=0.00292$, respectively.}
\end{figure}

We make the following observation. Given $\vps >0$, { and assume that $\delta$} is chosen slightly higher than the critical value in equation  \eqref{e10.1}.
In this case, the first eigenfunction will be (almost) supported by the square, while the second by the rectangle, provided the connecting tube is  thin enough.
 In such a way we can { construct} a sequence of domains  for which the first eigenfunctions localise (on the squares) while the second eigenfunctions do not localise.
 Assume now that $\delta$ is chosen slightly smaller than the critical value in equation   \eqref{e10.1}. In this case the second eigenfunction will be (almost) supported by the square,
 while the first one by the rectangle, provided the connecting tube is  thin enough.  In such a way we {  construct} a sequence of domains for which the second eigenfunctions localise (on the squares)
  while the first eigenfunctions do not localise. We conclude that there is no direct relationship between localisations of the first and second eigenfunctions respectively. {It is also possible to construct
   a sequence of domains for which both the first and the second Dirichlet eigenfunctions localise in $L^2$. Let $\Om_n$ be  a rhombus  with four sides of length $n$ and one diagonal of length $1$, with $ n \to +\infty$. The corresponding first eigenfunctions localise in $L^2$ (see  \cite{GJ98,GJ96,Je95} or Theorem \ref{the2} in this paper). The nodal line of the second eigenfunction is the shortest diagonal, so the second eigenfunction is a first eigenfunction on an elongating triangle and localises in $L^2$ as well.
   }

\section{Localisation of the first Dirichlet eigenfunction and Hardy's inequality \label{sec4}}

The results in this section are obtained under the hypothesis that the Dirichlet Laplacian satisfies the strong Hardy inequality. The mechanism for localisation is that the distance function is small on a very large set. The Hardy inequality implies that the boundary of this set is not thin, in terms of potential theory (see \cite{Anc}). This in turn implies that the eigenfunction is small on this large set and has most of its $L^2$ mass  on the complement.
\begin{definition}\label{def2}
The Dirichlet Laplacian $-\Delta$ acting in $L^2(\Omega)$ satisfies the strong Hardy
inequality, with constant $c_{\Omega}\in(0,\infty)$, if
\begin{equation}\label{Hardy}
\|\nabla w\|_2^2 \ge \frac{1}{c_{\Omega}} \int_{\Omega}\frac{w^2}{d_{\Omega}^2},
\quad \forall\, w \in C_c^\infty(\Omega),
\end{equation}
where $d_{\Omega}$ is the distance to the boundary function,
\begin{equation*}
d_{\Omega}(x)=\inf\{|x-y|:y\in \R^m\setminus\Omega\},\qquad x\in\Omega.
\end{equation*}
\end{definition}
Both the validity and applications
of inequalities like \eqref{Hardy} to spectral theory and partial differential equations have been
investigated in depth. See for example \cite{Anc},  \cite{EBD1}, \cite{EBD2}, \cite{EBD3} and \cite{EBD4}.
In particular it was shown in \cite[p.208]{Anc}, that for any proper simply connected open subset $\Omega$ in $\R^2$, inequality \eqref{Hardy} holds with
$c_{\Omega}=16$.
The following was proved in \cite {vdBK}.

\medskip

Let $(\Omega_n)$ be a sequence of open sets in $\R^m$ with $0<|\Omega_n|<\infty,\,n\in\N$, which satisfy \eqref{Hardy} with strong Hardy constants $c_{\Omega_n}$.
Suppose
\begin{equation}\label{e35.1}
\mathfrak{c}=\sup\{c_{\Omega_n}:n\in\N\}<\infty.
\end{equation}
\begin{itemize}
\item[\textup{(i)}]If  $(\eta_n)$ is a sequence of strictly positive real numbers such that
\begin{equation}\label{e34}
\lim_{n\rightarrow\infty}\frac{|\{d_{\Omega_n}\ge \eta_n\}|}{|\Omega_n|}=0,
\end{equation}
and
\begin{equation}\label{e35}
\lim_{n\rightarrow\infty}\frac{\eta_n^2|\Omega_n|}{\int_{\{d_{\Omega_n}\ge \eta_n\}}d_{\Omega_n}^2}=0,
\end{equation}
then $(v_{\Omega_n})$ localises along the sequence $(A_n)$ where $A_n=\{x\in\Omega_n:d_{\Omega_n}\ge \eta_n\}$.
\item[\textup{(ii)}]If any sequence $(A_n)$ of measurable sets, $A_n\subset\Omega_n,\,n\in \N,$ with
\begin{equation*}
\lim_{n\rightarrow\infty}\frac{|A_n|}{|\Omega_n|}=0,
\end{equation*}
satisfies
\begin{equation*}
\lim_{n\rightarrow\infty}\frac{\int_{A_n}d_{\Omega_n}^2}{\int_{\Omega_n}d_{\Omega_n}^2}=0,
\end{equation*}
then $(v_{\Omega_n})$ does not localise.
\end{itemize}

In \cite[Theorem 4]{vdBBK}, it was shown that if $(v_{\Om_n})$ localises in $L^1$ then $(u_{\Om_n})$ localises in $L^2$. This, together with the assertion  under (i) above, implies localisation of $(u_{\Om_n})$ provided \eqref{e34} and \eqref{e35} hold.
The following result asserts localisation of $(u_{\Om_n})$ under weaker assumptions.

\begin{theorem}\label{the1} Let $(\Omega_n)$ be a sequence of open sets in $\R^m$ with $0<|\Omega_n|<\infty,\,n\in\N$, which satisfies \eqref{e35.1}.
If there exists a sequence $(A_n)$ of measurable sets, $A_n\subset\Omega_n,\,n\in \N,$ with
\begin{equation}\label{e38}
\lim_{n\rightarrow\infty}\frac{|A_n|}{|\Omega_n|}=1,
\end{equation}
and which satisfies
\begin{equation}\label{e39}
\lim_{n\rightarrow\infty}\frac{\sup_{A_n}d_{\Om_n}}{\max_{\Om_n} d_{\Om_n}}=0,
\end{equation}
then for every $k \in \N$, $(u_{k,\Om_n})$ localises in $L^2$, where $u_{k,\Om_n}$ is an $L^2$-normalised eigenfunction corresponding to the $k$-th eigenvalue on $\Omega_n$.
\end{theorem}

\begin{proof}  For an open set with finite measure $\Omega$ and with a Hardy constant $c_\Om$, let  $u_{k,\Om}$ be a $k$-th Dirichlet eigenfunction normalised in $L^2(\Om)$. By Cauchy-Schwarz and \eqref{e30} we have for any measurable set $A\subset \Om$,
\begin{align}\label{e40}
\int_Au_{k,\Om}^2&\le \int_A \frac{u_{k,\Om}^2}{d_{\Om}^2}\big(\sup_Ad_{\Om}\big)^2\nonumber\\&
\le \big(\sup_Ad_{\Om}\big)^2\int_{\Om} \frac{u_{k,\Om}^2}{d_{\Om}^2}\nonumber\\&
\le c_{\Om}\big(\sup_Ad_{\Om}\big)^2\int_{\Om} |\nabla u_{k,\Om}|^2 \nonumber\\&
= c_{\Om}\lambda_k(\Om)\big(\sup_Ad_{\Om}\big)^2.
\end{align}

Since $\Om$ contains a ball of radius $\frac{1}{2}\sup_{\Om}d_{\Om}$ we have by monotonicity of Dirichlet eigenvalues that
\begin{equation}\label{e41}
 \lambda_k(\Om)\le 4\lambda_k(B_1) \big(\sup_{\Om}d_{\Om}\big)^{-2}.
\end{equation}
By \eqref{e40} and \eqref{e41}
\begin{equation*}
\int_Au_{k,\Om}^2\le 4 c_{\Om}\lambda_k(B_1) \bigg(\frac{\sup_Ad_{\Om}}{\sup_{\Om} d_{\Om}}\bigg)^2.
\end{equation*}
This implies the assertion in Theorem \ref{the1} since $\lim_{n\rightarrow\infty}\int_{A_n}u^2_{k,\Om_n}=0$, and so \\ $\lim_{n\rightarrow\infty}\int_{\Om_n\setminus A_n}u^2_{k,\Om_n}=1.$
By \eqref{e38}, $\lim_{n\rightarrow\infty}|\Om_n\setminus A_n|/|\Om_n|=0$. Hence $(u_{k,\Om_n})$ localises in $L^2$.
\end{proof}

Below we show that the hypotheses \eqref{e34}-\eqref{e35} imply those of Theorem \ref{the1}. Let $A_n=\{x\in\Om_n: d_{\Om_n}<\eta_n\}$. Hence \eqref{e34} implies \eqref{e38}.
Furthermore
\begin{align*}
\frac{\eta_n^2|\Omega_n|}{\int_{\{d_{\Omega_n}\ge \eta_n\}}d_{\Omega_n}^2}&\ge \frac{\eta_n^2|\Omega_n|}{\int_{\{d_{\Omega_n}\ge \eta_n\}} \sup_{\Om_n} d^2_{\Om_n}   }\nonumber\\&
\ge\frac{\eta_n^2}{\sup_{\Om_n} d^2_{\Om_n}}\nonumber\\&
\ge\bigg(\frac{\sup_{A_n}d_{\Om_n}}{\sup_{\Om_n} d_{\Om_n}}\bigg)^2.
\end{align*}
Hence \eqref{e35} implies \eqref{e39}.

\medskip

To prove that the hypotheses in Theorem \ref{the1} are strictly weaker than \eqref{e34}, \eqref{e35}, we have the following.
\begin{example}\label{exa1}{ Let $Q$ be the open unit square in $\R^2$ with vertices $(0,0),(1,0),(1,1)$ and $(0,1)$.
Let $0<\alpha<1,\,0<d <\infty$, and let $n\in\N$ be such that $dn^{-\alpha}<1$. For $a,b\in \R^2$ we denote by $L_{a,b}$ the closed line segment  with endpoints $a$ and $b$ respectively.
For $i=1,...,n-1,$ let $a_i=\big(\frac{i}{n},0\big),b_i=\big(\frac{i}{n},1-dn^{-\alpha}\big)$. The set
$$\Omega_{n,\alpha,d}\setminus \cup_{i=1}^{n-1}L_{a_i,b_i}$$ is open, simply connected with $|\Omega_{n,\alpha,d}|=1$. See Figure 3.  Hardy's inequality holds with $\mathfrak{c}=c_{\Omega_{n,\alpha,d}}=16$.
It was shown in \cite{vdBK} that
$(v_{\Omega_{n,\alpha,d}})$ localises in $L^1$ for $0<\alpha<\frac23$, and  does not localise for $\frac23<\alpha<1$. The proof that $(v_{\Omega_{n,\frac23,d}})$ $\kappa$-localises with $\kappa=\frac{d^3}{1+d^3}$ is quite involved  (see \cite{vdBK}). To prove that $(u_{\Omega_{n,\alpha,d}})$ localises in $L^2$ for all $0<\alpha<1$ we choose $A_n=\{x\in \Omega_{n,\alpha,d}:d_{\Omega_{n,\alpha,d}}<\frac{1}{2n}\}$. Then $\sup_{A_{n}}d_{\Om_{n,\alpha,d}}\le \frac{1}{2n}$, and
$$\max_{\Omega_{n,\alpha,d}}d_{\Omega_{n,\alpha,d}}\ge \frac12dn^{-\alpha}.$$}
Hence \eqref{e39} is satisfied. Also $|A_n|\ge 1-dn^{-\alpha}$ which implies \eqref{e38}. This implies localisation by Theorem \ref{the1}. We see that localisation takes place in a neighbourhood of the rectangle $\Omega_{n,\alpha,d}\cap\{x_2>1-dn^{-\alpha}\}$.
 \end{example}


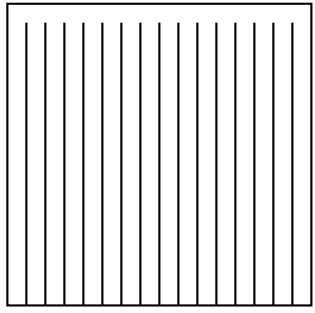
\begin{figure}
\centering
\begin{tikzpicture}
\draw[black, thick] (0,0) rectangle (4,4);
\draw[black, thick](0.25,0) -- (0.25,3.75);
\draw[black, thick](0.50,0) -- (0.50,3.75);
\draw[black, thick](0.75,0) -- (0.75,3.75);
\draw[black, thick](1,0) -- (1,3.75);
\draw[black, thick](1.25,0) -- (1.25,3.75);
\draw[black, thick](1.5,0) -- (1.5,3.75);
\draw[black, thick](1.75,0) -- (1.75,3.75);
\draw[black, thick](2,0) -- (2,3.75);
\draw[black, thick](2.25,0) -- (2.25,3.75);
\draw[black, thick](2.5,0) -- (2.5,3.75);
\draw[black, thick](2.75,0) -- (2.75,3.75);
\draw[black, thick](3,0) -- (3,3.75);
\draw[black, thick](3.25,0) -- (3.25,3.75);
\draw[black, thick](3.5,0) -- (3.5,3.75);
\draw[black, thick](3.75,0) -- (3.75,3.75);

\end{tikzpicture}
\caption{$\Omega_{n,\alpha,d}$ with $n-1$ parallel vertical line segments at distance $n^{-1}$ of length $1-dn^{-\alpha}$ { in the open unit square $Q$.}}
\label{fig2}
\end{figure}

\section{Localisation of the first Dirichlet eigenfunction for elongated horn-shaped regions \label{sec5}}

Below we obtain localisation results for sequences of sets in $\R^m$ which satisfy a monotonicity property in the $x_1$-direction along which elongation takes place. This monotonicity property is known in the literature as horn-shaped. The Dirichlet spectrum and eigenfunctions of horn-shaped open sets have been studied extensively in the non-compact setting in, for example, \cite{vdBD}, \cite{BavdB}, \cite{vdB1} and the  references there in. In \cite{vdBDPDBG} it was used to prove localisation for various examples such as \eqref{e43} mentioned above. We recall the set up and notation.

\begin{definition}\label{def3} Points in $\R^m$ are denoted by a Cartesian pair $(x_1,x')$ with $x_1\in \R,\, x'\in \R^{m-1}$. If $\Omega$ is an open set in $\R^m$, then we define its cross-section at $x_1$ by $\Omega(x_1)=\{x'\in \R^{m-1}:(x_1,x')\in \Omega\}$. A set $\Omega\subset\R^m$ is horn-shaped if it is non-empty, open, and connected, $x_1>x_2> 0$ implies $\Omega(x_1)\subset \Omega(x_2)$, and $x_1< x_2< 0$ implies $\Omega(x_1)\subset \Omega(x_2)$.
\end{definition}
Let $\Omega'$ be an open set in $\R^{m-1}$. Its first $(m-1)$-dimensional Dirichlet eigenvalue is denoted by $\mu(\Omega')$, and its $(m-1)$-dimensional Lebesgue measure is denoted by $|\Omega'|_{m-1}$.
For $a>0$ we let $a\Om'$ be the homothety of $\Om'$ by a factor $a$ with respect to that origin.

Let $-\infty<c_{-}\le0<c_{+}<\infty$. We consider the following class of monotone functions.
\begin{align*}
\mathfrak{F}&=\{f:[c_{-},c_{+}]\rightarrow [0,1],\, \textup{non-increasing, and right-continuous on}\, [0,c_{+}],\,\nonumber\\& \textup{non-decreasing, and left-continuous on}\, [c_{-},0],\,f(0)=1,\, f(x_1)<1\, \textup{for}\,x_1\ne 0\}.
\end{align*}
Given $f\in \mathfrak{F}$, let
$$f_n:[nc_-,nc_+]\rightarrow[0,1],\,f_n(x_1)=f(x_1/n), $$
let $\Om'\subset \R^{m-1}$ be a non-empty, open, bounded and convex set containing the origin, and let
$$\Om_{f_n,\Om'}=\{(x_1,x')\in\R^m: c_-n< x_1< c_+n,\, x'\in f(x_1/n)\Om'\}. $$

\begin{theorem}\label{the2}
\begin{enumerate}
\item[\textup{(i)}]If $f$ and $\Om'$ satisfy the hypotheses above, then $(u_{\Om_{f_n,\Om'}})$ localises in $L^2$.
{  \item[\textup{(ii)}]If $m=2$, if $f\in  \mathfrak{F}$ is concave such that $-c_-=c_+$, $f(x_1)=f(-x_1),0\le x_1<c_+$, and if $\Om'$ is an interval of length $1$ containing the origin, then $(u_{2,\Om_{f_n,\Om'}})$ localises in $L^2$.}
\end{enumerate}
\end{theorem}
\noindent The proof requires some lemmas which are given below.

The following is a generalisation of a two-dimensional bound. See \cite[Theorem 2]{MvdB}.
\begin{lemma}\label{lem3}
Let $\Om'$ be a non-empty open, bounded and convex set in $\R^{m-1}$ which contains the origin, let $f\in\mathfrak{F}$, and let
\begin{equation}\label{e50}
N^*=\min\{n\in \N:n\ge 1,\, f(c_+n^{-1/2})\ge 2^{-1}\}.
\end{equation}
If $n\ge N^*$ then
\begin{equation}\label{e50a}
\lambda_1(\Om_{f_n,\Om'})\le \mu(\Om')+\frac{\pi^2}{c_+^2n}+6\mu(\Om')(1-f(c_+n^{-1/2})).
\end{equation}
\end{lemma}
The proof is similar in spirit to the one in \cite[p.2095]{vdBNTV2}, and runs as follows.
\begin{proof}
Consider the cylinder $C_{f_n,\delta}$ with base $f_n(\delta)\Om'$ and height $\delta$ with $\delta<c_+n$.
By separation of variables
\begin{equation*}
\lambda_1(C_{f_n,\delta})= \frac{\pi^2}{\delta^2}+(f_n(\delta))^{-2}\mu(\Om').
\end{equation*}
By monotonicity of Dirichlet eigenvalues under inclusion,
\begin{align*}
\lambda_1(\Om_{f_n,\Om'})&\le \lambda_1(C_{f_n,\delta})\nonumber\\&
=\frac{\pi^2}{\delta^2}+(f(\delta/n))^{-2}\mu(\Om')\nonumber\\ &
=\mu(\Om')+\frac{\pi^2}{\delta^2}+(1-f(\delta/n))\frac{1+f(\delta/n)}{(f(\delta/n))^2}\mu(\Om').
\end{align*}
Choose $\delta=n^{1/2}c_+$ so that
$$\lambda_1({\Om_{f_n,\Om'}})\le\mu(\Om')+\frac{\pi^2}{c_+^2n}+ \frac{1+f(c_+n^{-1/2})}{(f(c_+n^{-1/2}))^2}(1-f(c_+n^{-1/2}))\mu(\Om').$$
Since $f$ is right-continuous at $0$, $N^*<\infty$. Furthermore since $f$ {  is non-increasing on $[0,c_+]$, and $(1+f)f^{-2}$ is non-decreasing} for $f>0$
we have by \eqref{e50} that
$$\frac{1+f(c_+n^{-1/2})}{(f(c_+n^{-1/2}))^2}\le 6,\,n\ge N^*.$$
\end{proof}

The Dirichlet heat kernel for an open set $\Om$ is denoted by $p_{\Om}(x,y;t),\,x\in\Om,\,y\in \Om,\,t>0$. If $|\Om|<\infty,$ then the spectrum of the Dirichlet Laplacian is discrete, and the corresponding Dirichlet heat kernel has an $L^2$-eigenfunction expansion given by
\begin{equation*}
p_{\Om}(x,y;t)=\sum_{j=1}^{\infty}e^{-t\lambda_j(\Omega)}u_{j,\Omega}(x)u_{j,\Omega}(y).
\end{equation*}
Recall that
\begin{equation*}
w_{\Om}(x;t)=\int_{\Om}\,dy\,p_{\Om}(x,y;t),
\end{equation*}
is the solution of the heat equation
\begin{equation*}
\Delta w=\frac{\partial w}{\partial t},\, x\in \Om, \, t>0,
\end{equation*}
with Dirichlet boundary condition
\begin{equation*}
w(\cdot;t)\in H_0^1(\Om;t),
\end{equation*}
and initial condition
\begin{equation*}
w(x;0)=1,\,x\in \Om.
\end{equation*}
The heat content for an open set $\Om\subset \R^m$ with finite Lebesgue measure at $t$ is given by
\begin{equation*}
Q_{\Om}(t)=\int_{\Om}\int_{\Om}dx\,dy\,p_{\Om}(x,y;t),
\end{equation*}
We denote by $\Om'$ an open set in $\R^{m-1}$. Its heat content (in dimension $m-1$) is also denoted by $Q_{\Omega'}(t)$.
\begin{lemma}\label{lem2}
If $|\Om|<\infty$, then
\begin{equation}\label{xs0}
Q_{\Om}(t)\le e^{-t\lambda_1(\Omega)}|\Om|,
\end{equation}
and
\begin{equation}\label{xs1}
\frac{1}{|\Om|}\bigg(\int_{\Om}u_{\Om}\bigg)^2\le\frac{e^{t\lambda_1(\Om)}}{|\Om|}Q_{\Om}(t).
\end{equation}
If $|\Om'|_{m-1}<\infty$, then
\begin{equation}\label{xs}
Q_{\Om'}(t)\le e^{-t\mu(\Om')}|\Om'|_{m-1}.
\end{equation}
\end{lemma}
\begin{proof}
It follows from Parseval's identity that
\begin{equation}
\label{xs2}
Q_\Om(t) = \sum_{j\in\N} e^{-t\lambda_j(\Om)} \left(\int_\Om u_{j,\Om}\right)^2
\leq e^{-t\lambda_1(\Om)} \sum_{j\in\N} \left(\int_\Om u_{j,\Om}\right)^2
= e^{-t\lambda_1(\Om)} |\Om|.
\end{equation}
This proves \eqref{xs0}. The first equality in \eqref{xs2} implies \eqref{xs1}. Inequality { \eqref{xs}} is the $(m-1)$-dimensional version of \eqref{xs0}.
\end{proof}

Let $(B(s),s\ge 0; \Pa_x,x\in \R^m)$ be Brownian motion
on $\R^m$ with generator $\Delta$. For $x\in \Omega$ we denote the first exit time of Brownian motion by
\begin{equation*}
T_{\Omega}=\inf\{s\ge 0: B(s)\notin \Omega\},
\end{equation*}
It is a standard fact that
\begin{equation}\label{e63}
w_{\Om}(x;t)=\Pa_x[T_{\Om}>t].
\end{equation}
So this gives
\begin{equation*}
\frac{1}{|\Om|}\bigg(\int_{\Om}u_{\Om}\bigg)^2\le \frac{e^{t\lambda_1(\Om)}}{|\Om|}\int_{\Om}dx\,\Pa_x[T_{\Om}>t].
\end{equation*}

The lemma below extends \cite[Theorem 5.3]{vdBD} to two-sided, horn-shaped regions.
\begin{lemma}\label{lem4} Let $\Om$ be horn-shaped in $\R^m$, and let both $|\Om|<\infty$, and $|\Om'|_{m-1}<\infty$. If $t>0$, then
\begin{equation}\label{e65}
Q_{\Om}(t)\le \int_{[c_-,c_+]}dx_1\,Q_{\Om(x_1)}(t)+4\bigg(\frac{t}{\pi}\bigg)^{1/2}Q_{\Om'}(t).
\end{equation}
\end{lemma}
\begin{proof}

It is convenient to define for horn-shaped sets,
\begin{equation}\label{e55}
\Om^-=\Om\cup\{(x_1,x')\in\R^m:x_1\le 0,x'\in\Om'\},
\end{equation}
and
\begin{equation*}
\Om^+=\Om\cup\{(x_1,x')\in\R^m:x_1\ge 0,x'\in\Om'\}.
\end{equation*}

For $x\in\Om,\,x_1>0$ we have by \eqref{e55}
\begin{equation*}
\Pa_x[T_{\Om}>t]\le \Pa_x[T_{\Om^{-}}>t].
\end{equation*}
Let $(B_1(s),\,s\ge 0)$ be $1$-dimensional Brownian motion in the $x_1$-direction, and let $(B'(s),\,s\ge 0)$ be an independent $(m-1)$-dimensional Brownian motion in the $x'$-plane.
Then, $B=(B_1,B')$. By solving the heat equation on $(-\infty,\xi)\times(0,\infty)$ with $\xi>0$, we have by \eqref{e63} and the preceding lines,
\begin{equation*}
\Pa_{0}[T_{(-\infty,\xi)}>t]=\int_{(0,\xi)}d\eta\,(\pi t)^{-1/2}e^{-\eta^2/(4t)}.
\end{equation*}
Hence the density of the random variable $\max_{0\le s\le t}B_1(s)$ with $B_1(0)=0$ is given by
\begin{equation*}
\rho(\xi;t)=(\pi t)^{-1/2}e^{-\xi^2/(4t)}1_{(0,\infty)}(\xi),
\end{equation*}
with a similar expression for $\min_{0\le s\le t}B_1(s).$
For $x\in\Om,\,x_1>0$,
\begin{align*}
\Pa_x&[T_{\Om^{-}}>t]\le \int_{\R^+}d\xi\,\rho(\xi;t)\Pa_{x'}[T_{\Om^{-}(x_1-\xi)}>t]\nonumber\\&
=\int_{(0,x_1)}d\xi\,\rho(\xi;t)\Pa_{x'}[T_{\Om^{-}(x_1-\xi)}>t]+\int_{(x_1,\infty)}d\xi\,\rho(\xi;t)\Pa_{x'}[T_{\Om^{-}(x_1-\xi)}>t]\nonumber\\&
=\int_{(0,x_1)}d\xi\,\rho(\xi;t)\Pa_{x'}[T_{\Om(x_1-\xi)}>t]+\int_{(x_1,\infty)}d\xi\,\rho(\xi;t)\Pa_{x'}[T_{\Om'}>t].
\end{align*}
We obtain that
\begin{align}\label{add}
&\int_{\Om\cap\{0\le x_1\le c_+\}}dx\,w_{\Om}(x;t)\nonumber\\&\le \int_{[0,c_+]}dx_1\int_{\Om(x_1)}dx'\,\int_{(0,x_1)}d\xi\,\rho(\xi;t)\Pa_{x'}[T_{\Om(x_1-\xi)}>t]\nonumber\\&
\hspace{3mm}+\int_{[0,c_+]}dx_1\,\int_{(x_1,\infty)}d\xi\,\rho(\xi;t)\int_{\Om'}dx'\,\Pa_{x'}[T_{\Om'}>t].
\end{align}
By Tonelli's Theorem we obtain that the first term in the right-hand side of \eqref{add} equals
\begin{align}\label{add1}
\int_{[0,c_+]}dx_1\int_{(0,x_1)}&d\xi\,\rho(\xi;t)\int_{\Om(x_1)}dx'\,\Pa_{x'}[T_{\Om(x_1-\xi)}>t]\nonumber\\&
\le \int_{[0,c_+]}dx_1\int_{(0,x_1)}d\xi\,\rho(\xi;t)\int_{\Om(x_1-\xi)}dx'\,\Pa_{x'}[T_{\Om(x_1-\xi)}>t]\nonumber\\&
=\int_{[0,c_+]}dx_1\int_{(0,x_1)}d\xi\,\rho(\xi;t)Q_{\Omega(x_1-\xi)}(t)\nonumber\\&
=\int_{[0,c_+]}dx_1Q_{\Om(x_1)}(t),
\end{align}
where we have used in the last equality that the integral of a convolution is the product of the integrals, and that the integral of a probability density equals $1$.
For the second term in the right-hand side of \eqref{add} we obtain by an integration by parts that
\begin{align}\label{add2}
\int_{[0,c_+]}dx_1\,&\int_{(x_1,\infty)}d\xi\,\rho(\xi;t)\int_{\Om'}dx'\,\Pa_{x'}[T_{\Om'}>t]\nonumber\\&
\le \int_{[0,\infty)}dx_1\,\int_{(x_1,\infty)}d\xi\,\rho(\xi;t)\int_{\Om'}dx'\,\Pa_{x'}[T_{\Om'}>t]\nonumber\\&
=\bigg(\frac{4t}{\pi}\bigg)^{1/2}Q_{\Om'}(t).
\end{align}
By \eqref{add1} and \eqref{add2}, we have
\begin{equation}\label{add3}
\int_{\Om\cap\{0\le x_1\le c_+\}}dx\,w_{\Om}(x;t)\le \int_{[0,c_+]}dx_1Q_{\Om(x_1)}(t)+\bigg(\frac{4t}{\pi}\bigg)^{1/2}Q_{\Om'}(t).
\end{equation}
Similarly,
\begin{equation}\label{add4}
\int_{\Om\cap\{c_-\le x_1\le 0\}}dx\,w_{\Om}(x;t)\le \int_{[c_-,0]}dx_1Q_{\Om(x_1)}(t)+\bigg(\frac{4t}{\pi}\bigg)^{1/2}Q_{\Om'}(t).
\end{equation}
Adding the contributions from \eqref{add3} and \eqref{add4} gives \eqref{e65}. Note that the hypotheses on $|\Om|$ and $|\Om'|_{m-1}$ guarantee that the right-hand side of \eqref{e65} is finite for all $t>0$.
\end{proof}

\noindent{\it Proof of Theorem \ref{the2}.}
{ (i)} Since $f\in \mathfrak{F}$, and $\Om'$ is convex containing the origin, $\Om_{f,\Om'}$ is horn-shaped.
By Lemma \ref{lem2} applied to the $(m-1)$-dimensional set $f_n(x_1)\Om'$ we have
\begin{align}\label{e72}
Q_{\Om_{f_n,\Om'}(x_1)}(t)&=Q_{f(x_1/n)\Om'}(t)\nonumber\\&
\le (f(x_1/n))^{m-1}|\Om'|_{m-1}e^{-t\mu(\Om')(f(x_1/n))^{-2}}\nonumber\\&
\le |\Om'|_{m-1}e^{-t\mu(\Om')(f(x_1/n))^{-2}}.
\end{align}
Furthermore
\begin{equation}\label{e73}
|\Om_{f_n,\Om'}|=n|\Om_{f,\Om'}|.
\end{equation}
By \eqref{xs1}, \eqref{e65}, and \eqref{e73}, we have
\begin{align}\label{e73a}
&\frac{1}{|\Om_{f_n,\Om'}|}\bigg(\int_{\Om_{f_n,\Om'}}u_{\Om_{f_n,\Om'}}\bigg)^2\nonumber\\&\le \frac{e^{t\lambda_1(\Om_{f_n,\Om'})}}{n|\Om_{f,\Om'}|}\Bigg(\int_{[nc_-,nc_+]}dx_1\,Q_{\Om_{f_n,\Om'}(x_1)}(t)+4\bigg(\frac{t}{\pi}\bigg)^{1/2}Q_{\Om'}(t)\bigg)\nonumber\\&
\le\frac{e^{t\lambda_1(\Om_{f_n,\Om'})}|\Om'|_{m-1}}{n|\Om_{f,\Om'}|}\Bigg(\int_{[nc_-,nc_+]}dx_1e^{-t\mu(\Om')(f(x_1/n))^{-2}}+4{  e^{-t\mu(\Omega')}}\bigg(\frac{t}{\pi}\bigg)^{1/2}\Bigg)\nonumber\\&
=\frac{e^{t\lambda_1(\Om_{f_n,\Om'})}|\Om'|_{m-1}}{|\Om_{f,\Om'}|}\Bigg(\int_{[c_-,c_+]}dx_1e^{-t\mu(\Om')(f(x_1))^{-2}}+\frac{4}{n}e^{-t\mu(\Om')}\bigg(\frac{t}{\pi}\bigg)^{1/2}\Bigg),
\end{align}
where we have used \eqref{e72} and \eqref{xs} in the third line above.
By \eqref{e50a} and \eqref{e73a} we have
\begin{align}\label{e73b}
\frac{1}{|\Om_{f_n,\Om'}|}&\bigg(\int_{\Om_{f_n,\Om'}}u_{\Om_{f_n,\Om'}}\bigg)^2\le \frac{e^{t\big(\frac{\pi^2}{nc_+^2}+6\mu(\Om')(1-f(n^{-1/2}c_+))\big)}|\Om'|_{m-1}}{|\Om_{f,\Om'}|}\nonumber\\&\times\bigg(\int_{[c_-,c_+]}dx_1e^{t\mu(\Om')(1-(f(x_1))^{-2})}+\frac{4}{n}\bigg(\frac{t}{\pi}\bigg)^{1/2}\bigg).
\end{align}
To complete the proof we choose
\begin{equation}\label{e76}
t=t_n=\bigg(\frac{\pi^2}{nc_+^2}+6\mu(\Om')(1-f(n^{-1/2}c_+))\bigg)^{-1}.
\end{equation}
Substituting this into \eqref{e73b} gives
\begin{align}\label{e77}
&\frac{1}{|\Om_{f_n,\Om'}|}\bigg(\int_{\Om_{f_n,\Om'}}u_{\Om_{f_n,\Om'}}\bigg)^2\le\frac{e|\Om'|_{m-1}}{|\Om_{f,\Om'}|}\nonumber\\&
\hspace{9mm}\times\bigg(\int_{[c_-,c_+]}dx_1\,e^{t_n\mu(\Om')(1-(f(x_1))^{-2})}+\bigg(\frac{16t_n}{\pi n^2}\bigg)^{1/2}\bigg).
\end{align}
The integrand in the first term in the right-hand side of \eqref{e77} side is bounded by $1,$ and is integrable on $[c_-,c_+]$. This term goes to $0$ as $n\rightarrow\infty$ by Lebesgue's dominated convergence theorem since $t_n\rightarrow \infty$ and $1-(f(x_1))^{-2}<0$ for all $x_1\ne 0$. The second term is $O(n^{-1/2})$ by \eqref{e76}.
Localisation in $L^2$ follows by Lemma \ref{lem0}. This proves (i).

(ii) Under the hypotheses of (ii), the sets $\Om_{f_n,\Om'}$ are convex and symmetric with respect to the vertical axis. Following the results of Jerison \cite[Theorem B]{Je95}, and Grieser and Jerison \cite[Theorem 1]{GJ96} the  second eigenfunction has to be odd in the $x$ variable, hence to have the nodal line on the vertical axis.
Indeed, assume there is a second eigenfunction  $u_{2,n}$  which is not odd in the $x$ variable. Then  $v(x,y)=u_{2,n}(x,y)+u_{2,n}(-x,y)$ is a non trivial second eigenfunction which is even in the $x$ variable, thus having a nodal line symmetric about the vertical axis.
Following \cite[Theorem 1]{GJ96} the nodal line is contained in a vertical strip of width of order $\frac 1n$. There are two possibilities: (i)
the nodal  line intersects the upper and lower boundary and, from symmetry, we get more than two nodal domains, thus ending up with a contradiction, (ii) the nodal line intersects only one of the boundaries enclosing a nodal domain with first eigenvalue of order $n^2$ contradicting that the eigenvalues do converge to $\pi^2$.

\hspace*{\fill }$\square $

\section{Example of $\kappa$-localisation for Neumann eigenfunctions\label{sec3}}
In this section we construct a sequence of simply connected, planar, polygonal domains for which the corresponding sequence of first Neumann eigenfunctions $\kappa$-localises in $L^2$.

Localisation of the first  Neumann eigenfunction has been implicitly noted in \cite[Theorem 4.1]{AHH91} based on the following (Courant-Hilbert) example, with the geometry similar to Fig  1.
Let $\eta >0$ and define for $\vps>0$ small

$${ R}=(-1, 0) \times (-1, 1),$$

$$T_{\vps,\eta}=[0,\vps] \times (- \vps^\eta, \vps^\eta),$$

$${ S}_\vps=(\vps, 2\vps)\times (-\vps, \vps),$$
and

  $$\Om_\vps:= { R}\cup T_{\vps,\eta}\cup { S}_\vps.$$

Consider the Neumann eigenvalue problem in $\Om_\vps$, and denote by $\mu_1(\Om_\vps)$ the first non-zero Neumann eigenvalue of the Laplace operator. Let $u_\vps$ {  be} a first $L^2$-normalised corresponding eigenfunction. The following result was proved in \cite[Theorem 4.1]{AHH91} (also \cite{Ar95}): let $\eta >3$ and  $\vps \rightarrow 0$, then $\mu_1(\Om_\vps) \rightarrow 0$ and ${  \int_{S_\vps}} u_\vps^2 dx \rightarrow 1$. In other words, the sequence of the first Neumann eigenfunctions localises.

We introduce the following geometry. For every small $ \theta  >0$ and $\delta $ in a neigbourhood of $0$, we define the following sets.
The open rectangle
$${ S}=(-1,0)\times (-1, 1)\subset \R^2,$$
with $\mu_1({ S})=  \frac{\pi^2}{4} $ simple, and
the rectangle
$${ R}_{\delta,\theta}=[0,1+\delta) \times (-\theta, \theta).$$
Note that the first eigenvalue of the segment of length $1$ and with Dirichlet boundary conditions at one vertex and Neumann boundary conditions at the opposite vertex is equal to $\frac{\pi^2}{4} $, and is also simple.

Let
$$\Om_{\delta, \theta} = { S} \cup { R}_{\delta,\theta},$$
and let $u^1_{\delta, \theta}$ {  be a} first eigenfunction.
See Figure   4.
\medskip

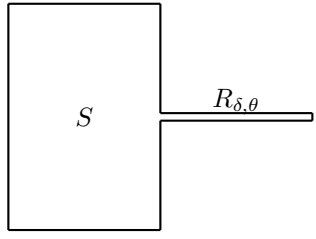
\begin{figure}
\centering
\begin{tikzpicture}
\draw[black, thick](0,-0.5)--(2,-0.5);
\draw[black, thick](2,-0.5)--(2,0.95);
\draw[black, thick](2,1.05)--(2,2.5);
\draw[black, thick](2,2.5)--(0,2.5);
\draw[black, thick](2,0.95)--(4,0.95);
\draw[black, thick](4,0.95)--(4,1.05);
\draw[black, thick](4,1.05)--(2,1.05);
\draw[black, thick](0,2.5)--(0,-0.5);
\node at (1,1.0)    {$S$};
\node at (3,1.2)    {$R_{\delta,\theta}$};
\end{tikzpicture}
\caption{$S\cup R_{\delta,\theta}$}
\label{fig3}
\end{figure}
\medskip
\begin{theorem}
Let $\kappa\in (0,1)$ be fixed.  There exists a sequence of sets of the form $\Om_{\delta, \theta}$ for which the first Neumann eigenfunction $\kappa$-localises.
\end{theorem}
Following  Jimbo \cite{Ji89} and Arrieta \cite{Ar95}, when $\delta \not=0$ is fixed and $\theta \ra 0$ the eigenvalues of the Neumann Laplacian on $\Om_{\delta, \theta}$ converge to the union of eigenvalues of the segment of length $1+\delta$ and mixed Dirichlet-Neumann boundary conditions and the Neumann spectrum of $S$.

The idea is to identify suitable pairs $(\delta_n, \theta_n) \ra (0,0)$ either with double first non-zero eigenvalue or with a simple first non-zero eigenvalue having an eigenfunction with balanced mass between $ S$ and ${ R}_{\delta_n,\theta_n}$. Both situations will lead to  $\kappa$-localisation.

In Lemma \ref{lem1} below, we give some information of the behaviour of a sequence of eigenfunctions on $\Om_{\delta, \theta} $ when $(\delta, \theta) \ra (0,0)$.
For further details concerning the spectrum on these kinds of geometries we refer to \cite{Ar95}.

\begin{lemma}\label{lem1}
Let $(\delta_n, \theta_n) \ra (0,0)$ and $(u_n, \mu_n)$
be an eigenpair on $\Om_n:= \Om_{\delta_n, \theta_n}$, such that $\int_{\Om_n} u_n^2  =1$ and $\limsup_\nif \mu_n< +\infty$. Then, there exist $\mu \ge 0$ and a subsequence (still denoted with the same index) such that
the following hold.
 \begin{enumerate}
 \item[\textup{(i)}] $u_n|_S \rau u$, weakly in $H^1(S)$, strongly in $L^2(S)$ with $\int _ S u dx =0$, and
 $$
\begin{cases}
-\Delta u = \mu  u \mbox { in } S,\\
\frac{\partial u}{\partial n} = 0  \mbox { on } \partial S.
\end{cases}
$$
 \item[\textup{(ii)}] Denoting $v_n(x,y):= \sqrt{\theta_n} u_n(-x^-+\frac{x^+}{1+\delta_n}, \theta_n y)$, $\tilde S=(-1,1)\times (-1,1)$, we have $v_n \rau v$ weakly in $H^1(\tilde S)$, strongly in $L^2(\tilde S)$, with $v(x,y)=v(x)\in H^1(-1,1)$ and
 $$
\begin{cases}
-v'' = \mu v \mbox { in } (0,1),\\
v(0) = 0, v'(1)=0.
\end{cases}
$$
\end{enumerate}
\end{lemma}
Note that $u$ or $v$ in the above may be the $0$-function.
\begin{proof}
For a subsequence we can assume $\mu_n \ra \mu$ and $u_n|_\S \rau u$, weakly in $H^1( S)$. Let  $\vphi \in H^1_{loc}(\R^2)$.  Note that
\begin{align*}\Big|\int_{R_n} u_n \vphi \Big|&+ \Big|\int_{R_n} \nabla u_n \nabla \vphi \Big|\nonumber\\&\le \|u_n\|_{L^2(\Om_n)} \Big(\int_{R_n}\vphi^2 \Big)^\frac 12 + \|\nabla u_n\|_{L^2(\Om_n)} \Big(\int_{R_n}|\nabla \vphi|^2 \Big)^\frac 12 \ra 0.
\end{align*}
This implies, in particular, $\int_S u_n \ra 0$ and hence $\int_S u =0$.  Taking $\vphi|_{\Om_n}$ as a test function in $H^1(\Om_n)$ we get
$$\int_{\Om_n} \nabla u_n \nabla \vphi = \mu_n \int_{\Om_n} u_n\vphi.$$
Splitting the sums  over $\Om_n=S \cup \Rr_n$, and using the weak convergence in $H^1(S)$ we get
$$\int_{S} \nabla u \nabla \vphi = \mu \int_{S} u\vphi.$$
Since $H^1_{loc} (\R^2)|_S$ coincides with $H^1(S)$, Lemma \ref{lem1} part (i) is proved.

To prove Lemma \ref{lem1} part (ii), we note that
$$\int_{\tilde S} v_n^2 \le 1+|\delta_n|, \int_{\tilde S} \Big( \frac{\partial v_n}{\partial x} \Big)^2 \le  (1+2|\delta_n|)\mu_n,  \int_{\tilde S} \Big( \frac{\partial v_n}{\partial y} \Big)^2 \le  (1+|\delta_n|) \theta^2_n.$$
Then, for a subsequence, $(v_n)$,  $v_n\rau v$ weakly in $H^1(\tilde S)$ with $\frac{\partial v}{\partial y}=0$ in $\tilde S$. So the function $v$ depends only on the variable $x$. Moreover, $v$ is continuous and $v=0$ on $(-1,0]$. This is a consequence of the trace theorem on $(-1,0)\times \{0\}$ applied to $u_n$ giving that
$\int_{-1}^0 u_n(x,0) ^2 dx $
is bounded. This implies that $\sqrt{\theta_n} u_n(\cdot, 0)$ converges strongly to $0$ on $(-1,0)$. This also implies that the convergence is strong in $L^2(\tilde S)$.

Taking a test function $\varphi \in H^1(0,1)$ with $\varphi (0)=0$, that we extend by zero on $(-1,0)$ and constant in $y$ on $(-1,1)$ in the equation satisfied by $u_n$, we get
$$\int_{R_n} \nabla u_n \nabla \varphi = \mu_n \int_{R_n} u_n\vphi,$$
and in terms of $v_n$
$$\int_{(0,1) \times (-1,1)}  \partial_x v_n \partial _x \vphi= \mu_n (1+\delta_n) \int_{(0,1) \times (-1,1)}  v_n \vphi,$$
that we pass to the limit to get the equation.
 \end{proof}

\begin{proof}

Fix $\kappa\in (0,1)$.  Let $\delta_1 >0$. Following \cite{Ar95} we know that for $\theta \ra 0$
$$\mu_1(\Om_{\delta_1, \theta}) \ra \Big (\frac{\pi}{2+2 \delta_1} \Big)^2, \mu_2(\Om_{\delta_1, \theta}) \ra \frac{\pi^2}{4},$$
with convergence of eigenfunctions given by the preceding Lemma. Hence $\int_{R_{\delta_1, \theta}} (u^1_{\delta_1, \theta})^2 \ra 1$.

At the same time
$$\mu_2(\Om_{-\delta_1, \theta}) \ra \Big (\frac{\pi}{2-2 \delta_1} \Big)^2, \mu_1(\Om_{-\delta_1, \theta}) \ra \frac{\pi^2}{4}.$$
Hence $\int_{R_{\delta_1, \theta}} (u^1_{-\delta_1, \theta})^2 \ra 0$.

We choose $\theta$ small enough such that
$$\int_{R_{\delta_1, \theta}} (u^1_{\delta_1, \theta})^2 \ge \frac {1+\kappa}{2}  \mbox{ and  } \int_{R_{\delta_1, \theta}} (u^1_{-\delta_1, \theta})^2 \le \frac \kappa2.$$
For this value of $\theta$, denoted by $\theta_1$, we vary  $\delta$ continuously from $-\delta_1$ to $\delta_1$. The spectrum of the Neumann Laplacian varies continuously along this trajectory, and the eigenfunctions corresponding to simple eigenvalues are continuous. In particular if the first eigenvalue is always simple, then the mass of the corresponding eigenfunction varies continuously on $S$ (and its complement).

There are two situations: either the first eigenvalue is simple along the entire trajectory, or not. In the latter case, we stop at the point when the eigenvalue becomes double.

We now repeat this procedure, taking $\delta_2= \delta_1/2$, choosing $\theta _2 \le \theta_1/2$, and so on.
In this way we find either a sequence of sets $(\Om_n)$ either with simple first eigenvalues and with balanced mass $1-\kappa$ on $S$ and $\kappa$ on $R_n$, or a sequence of sets $(\Om_n)$ with double first eigenvalues.

If the first situation occurs, the sequence of eigenfunctions $\kappa$-localises. Indeed, on $S$ the sequence converges to a first eigenfunction of $S$ which has the mass $1-\kappa$ and no localisation can occur on $S$. For   $A_n \sq S$ we have
$$\int_{A_n} u_n^2 \le |1_{A_n} | _{L^2} |u_n^2| _{L^2} \ra 0,$$
from the continuous injection $H^1(S) \sq L^4(S)$.

If the second situation occurs, let us denote $u_n^1, u_n^2$ two normalised $L^2$-orthogonal eigenfunctions corresponding to the first (double) eigenvalue. We follow the masses of the eigenfunctions: assume (for a subsequence) that
$$\int_{S} (u_n^1)^2 dx \ra a, \quad \int_{S}  (u_n^2)^2 dx \ra b.$$
If both $a\ne 0$ and $b\ne 0$, then we consider the weak $H^1(S)$-limits of $u_n^1|_S$ and  $u_n^2|_S$, denoted $u^1,u^2$, respectively. Both of them are non-zero eigenfunctions corresponding to the first eigenvalue on $S$. This being simple, there exists $\lb\in \R$ such that $u^1+\lb u^2=0$. This implies that the sequence given by $\tilde u_n=\frac{1}{\sqrt{1+\lb^2}}(u_n^1+\lb u_n^2)$ is a sequence of normalised first eigenfunctions converging to $0$ on $S$. In other words, we can assume that $a=0$ and relabel $u_n^1=\tilde u_n$.

A similar argument applied to $R_{\delta_n, \theta_n}$, gives that $b=1$. Indeed, if $b\not=1$, then the sequences $v_n^1,v_n^2$ constructed in Lemma  \ref{lem1} part (ii) would converge to a non-zero first eigenfunction on the segment $(0,1)$, so that the previous argument can be used again.

Since we know now that for suitable sequences of eigenfunctions we have $a=0,b=1$, we consider the sequence
$\kappa  u_n^1+\sqrt{1-\kappa}u_n^2$
of normalised first eigenfunctions on $\Om_n$ which $\kappa$-localises.
\end{proof}
The data in Figure  5 have been obtained with the MATLAB PDE toolbox, and illustrate the mass distribution of the first eigenfunction.
\begin{figure}[!ht]\label{fig2vbb}
\includegraphics[width=4cm]{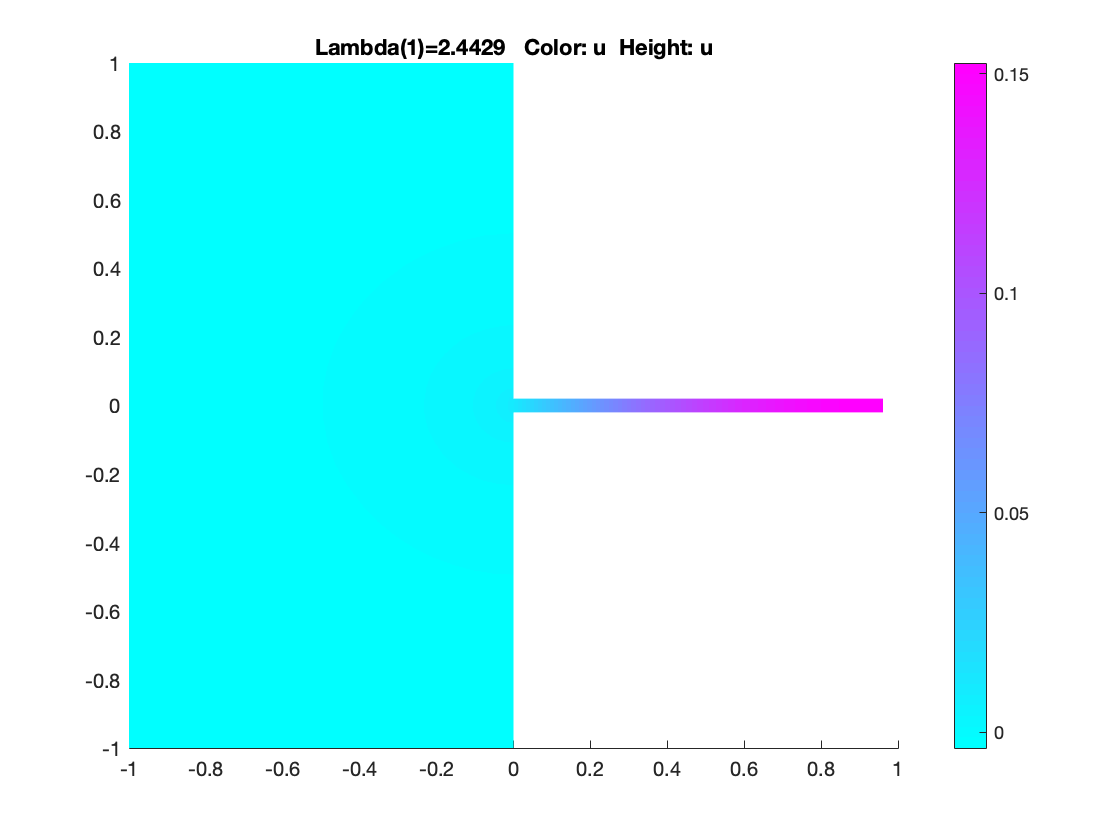}\hskip 0cm
\includegraphics[width=4cm]{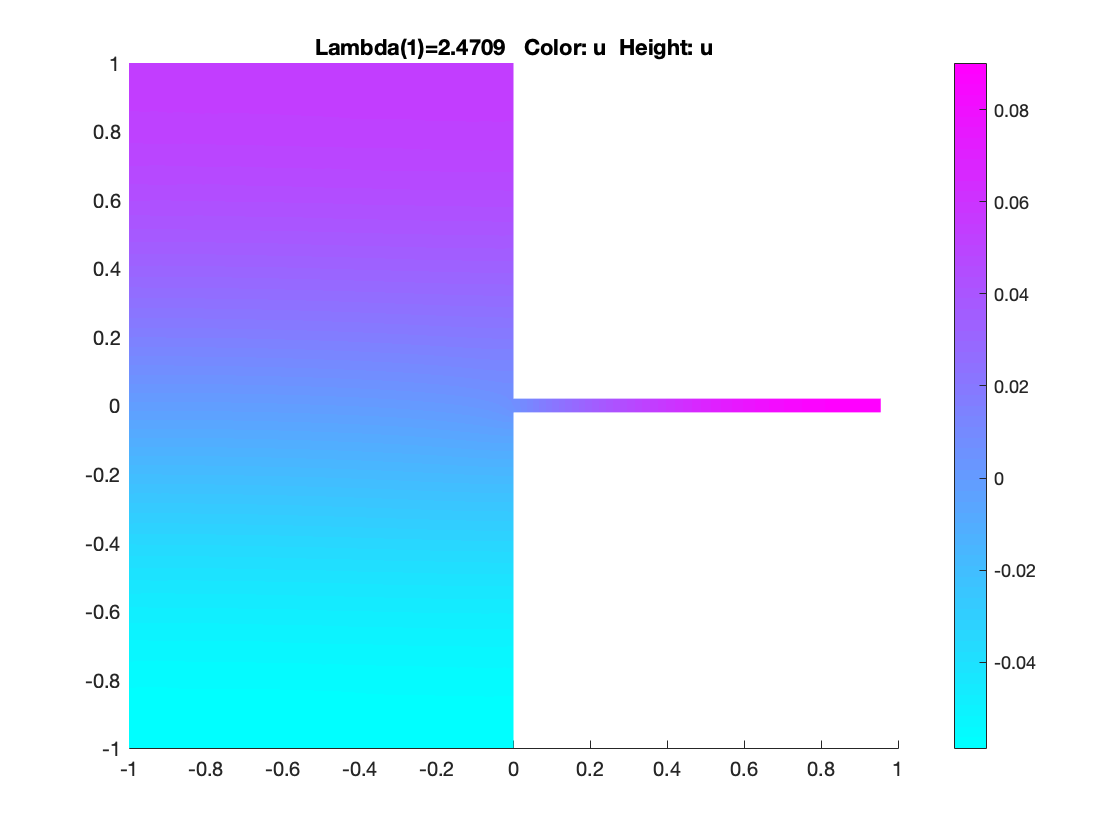}\hskip 0cm
\includegraphics[width=4cm]{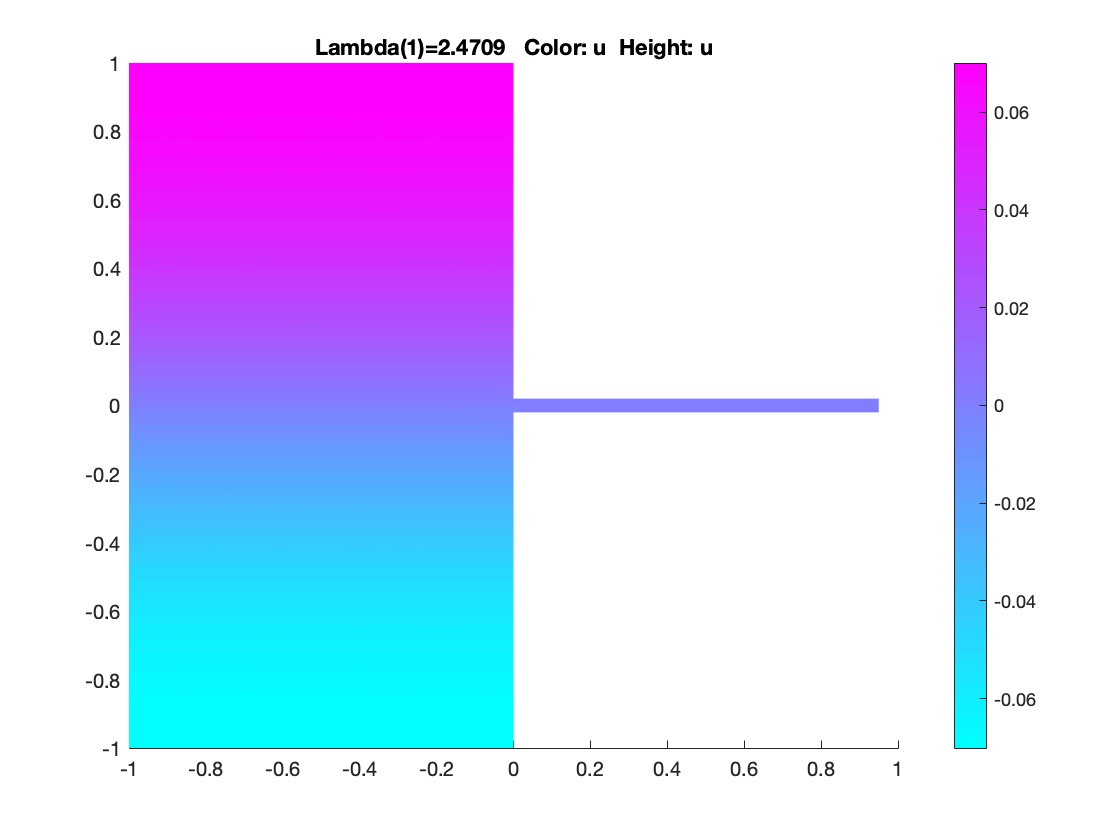}

\includegraphics[width=4cm]{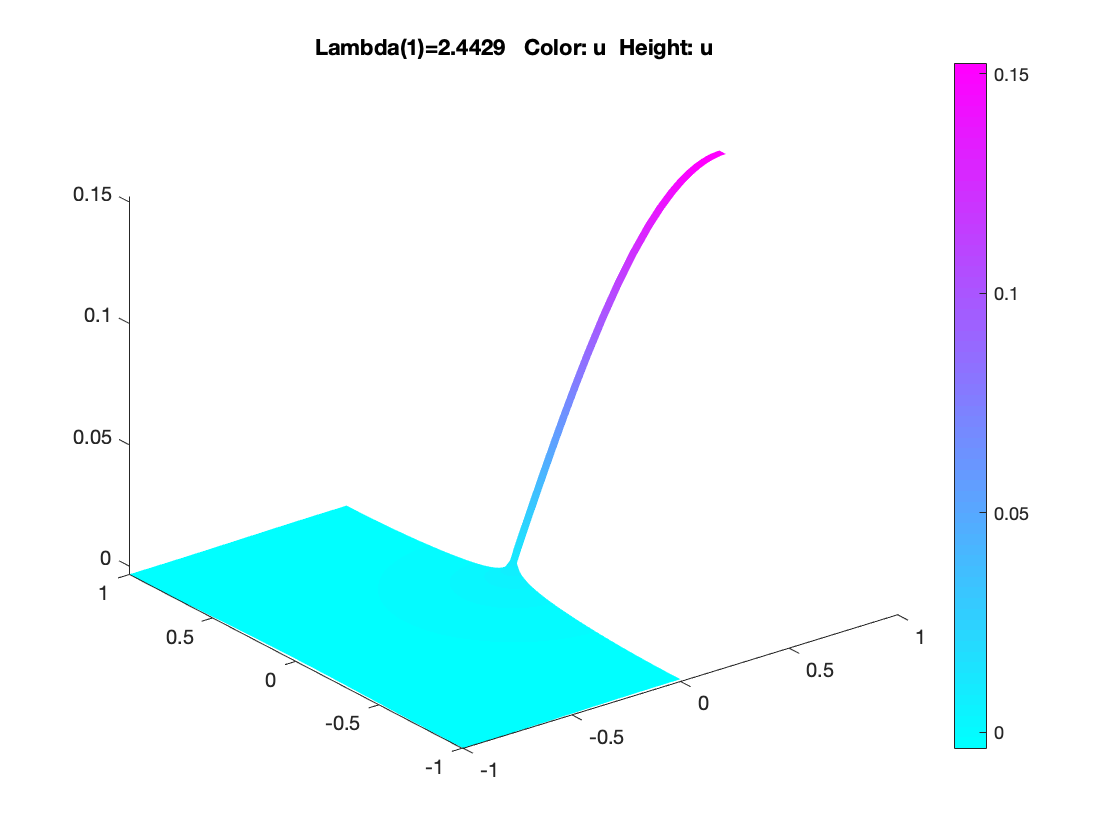}\hskip 0cm
\includegraphics[width=4cm]{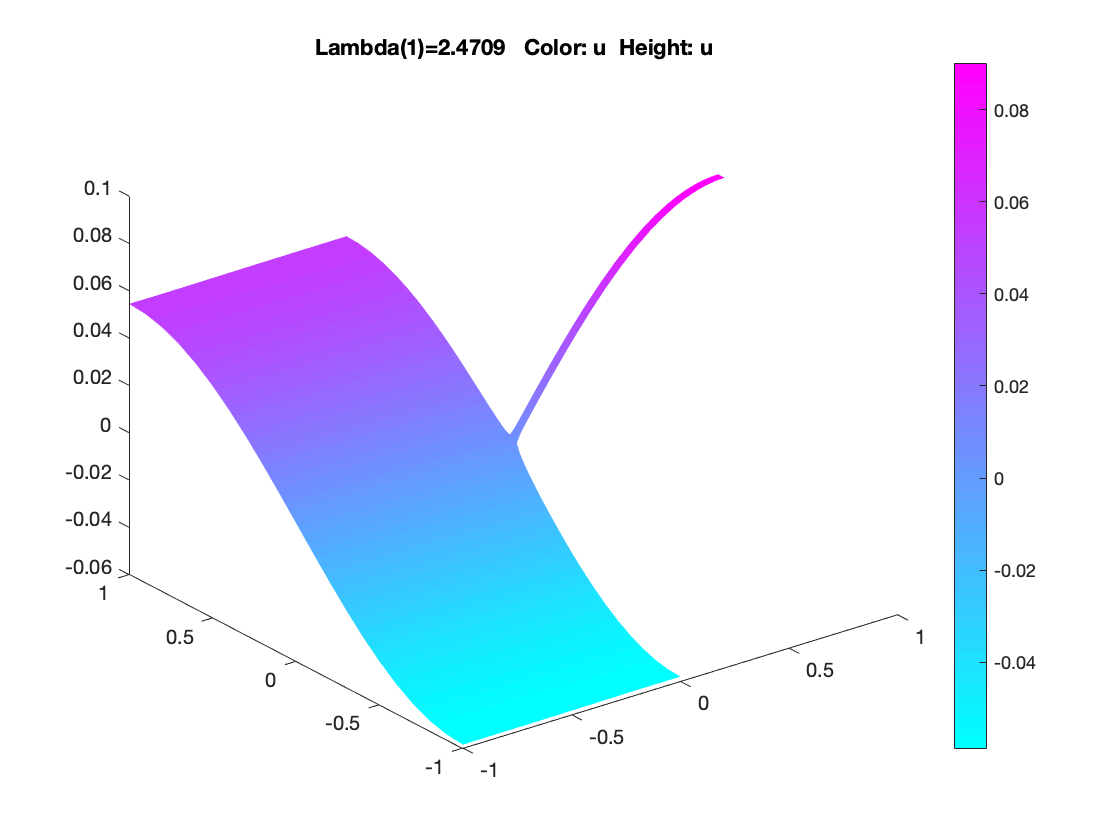}\hskip 0cm
\includegraphics[width=4cm]{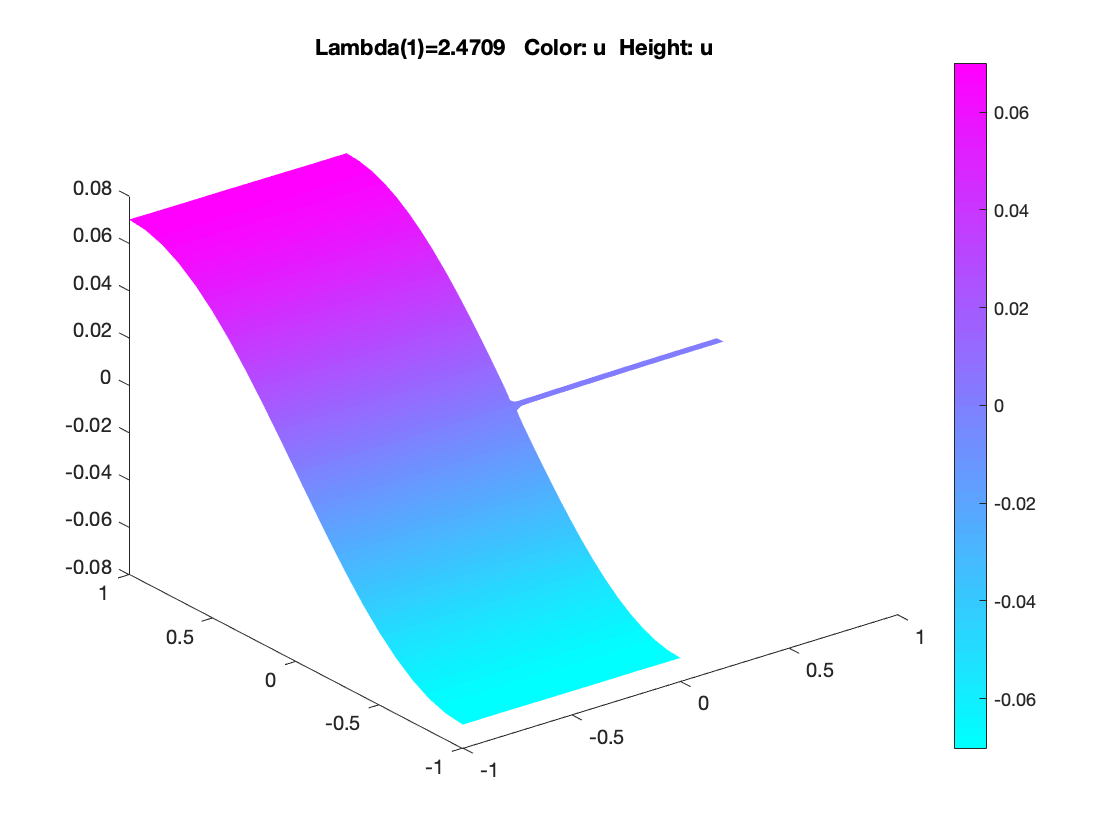}
\caption{The graph of $u^1_{\delta, \theta}$, from localisation to non localisation, when perturbing the length of the thin rectangle: $\theta=0.02$ and $\delta=-0.039$, $\delta=-0.04491$, $\delta=-0.05$, respectively.}
\end{figure}


\end{document}